\documentclass[12pt]{article}

\usepackage{amsfonts}

\usepackage{graphicx}
\usepackage{amsmath}
 \usepackage{color}

\newcounter{thm}[section]
\newcounter{appen}

\newtheorem{theor}[thm]{Theorem}
\newtheorem{cor}[thm]{Corollary}
\newtheorem{proposition}[thm]{Proposition}

\newtheorem{assumption}[appen]{Assumption}
\newtheorem{lemma}[appen]{Lemma}

\newenvironment{proof}[1][Proof]{\textbf{#1.} }{\ \rule{0.5em}{0.5em}}
\begin{document}

\title{A semiparametric single-index estimator for a class of estimating equation models}
\author{Marian Hristache\thanks{ CREST (Ensai), email: marian.hristache@ensai.fr} \ \ \
Weiyu Li\thanks{Corresponding author. CREST (Ensai), email: liweiyu84@gmail.com} \ \ \
 Valentin Patilea\thanks{CREST (Ensai), email: valentin.patilea@ensai.fr. Valentin Patilea gratefully acknowledges support from
the research program New Challenges for New Data of Genes, LCL and Fondation de Risque.}}
\date{}
\maketitle

\begin{abstract}
We propose a two-step pseudo-maximum likelihood procedure for semiparametric single-index regression models where the conditional variance is a known function of the regression and an
additional parameter. The Poisson single-index regression with multiplicative unobserved heterogeneity is an example of such models. Our procedure is based on linear exponential densities with nuisance parameter. The pseudo-likelihood criterion we use contains a nonparametric estimate of the index regression and therefore a rule for choosing the smoothing parameter is needed. We propose an automatic and natural rule based on the joint maximization of the pseudo-likelihood with respect to the index parameter and the smoothing parameter. We derive the asymptotic properties of the semiparametric estimator of the index parameter and the asymptotic behavior of our `optimal' smoothing parameter. The finite sample performances of our methodology are analyzed using simulated and real data.


\end{abstract}

\vskip -1.5cm

\vskip -1.5cm

\newpage

\section{Introduction}

\setcounter{equation}{0}


In this paper we consider semiparametric models defined by conditional mean and conditional variance estimating equations. Models defined by estimating equations for the first and second order conditional moments are widely used in applications. See, for instance, Ziegler (2011) for a recent reference. Here we consider a model that extends the framework considered by Cui, H\"{a}rdle and Zhu (2011).

To provide some insight on the type of models we study, consider the following semiparametric extension of the classical Poisson regression model with unobserved heterogeneity: the observed variables are $\left( Y,Z^{T}\right) ^{T}$ where $Y$ denotes the count variable and $Z$ is
the vector of $d$ explanatory variables. Let $r\left( t;\theta \right) =E \left( Y\mid Z^{T}\theta =t\right) .$ We assume that there exists $\theta_{0}\in \mathbb{R}^{d}$ such that
\begin{equation*}
E\left( Y\mid Z\right) =E\left( Y\mid Z^{T}\theta _{0}\right) = r\left( Z^{T}\theta _{0};\theta _{0}\right) .
\end{equation*}
The parameter $\theta _{0}$ and the function $r$ are unknown. Given $Z$ and an unobserved error term $\varepsilon ,$ the variable $Y$ has a Poisson law of mean $r\left( Z^{T}\theta _{0};\theta _{0}\right) \varepsilon .$ If $E \left(  \varepsilon \mid Z\right)  =1$ and $Var\left(  \varepsilon \mid Z \right) =\sigma ^{2},$ then
\begin{align}
Var\left( Y\mid Z\right) &=Var\left( E\left(  Y\mid Z,\varepsilon \right) \mid Z\right) +E\left( Var\left( Y\mid Z,\varepsilon \right)  \mid Z\right)
\notag \\
&=r\left( Z^{T}\theta _{0};\theta _{0}\right) \left[ 1+\sigma ^{2}r\left( Z^{T}\theta _{0};\theta _{0}\right) \right] .  \label{mom2}
\end{align}
This model is a semiparametric single-index regression model (e.g., Powell, Stock and Stoker (1989), Ichimura (1993), H\"{a}rdle, Hall and Ichimura (1993), Sherman (1994b)) where a second order conditional moment is specified as  a nonlinear function of the conditional mean and an additional unknown parameter. This extends the framework of Cui, H\"{a}rdle and Zhu (2011) where the conditional variance of the response is proportional to a given function of the conditional mean.

Our first contribution is to propose a new semiparametric estimation procedure for single-index regression which  incorporates the additional information on the conditional variance of $Y$.
For this we extend the
quasi-generalized pseudo maximum likelihood method introduced by Gouri\'{e}roux, Monfort and Trognon (1984a, 1984b) to a semiparametric framework. More precisely, we propose to estimate $\theta _{0}$ and the function $r(\cdot)$ through a two-step pseudo-maximum likelihood (PML) procedure based on \emph{linear exponential families} \emph{with nuisance parameter} densities. Such densities are parameterized by the mean $r$ and a nuisance parameter that can be recovered from the variance. Although we use a likelihood type criterion, no conditional distribution assumption on $Y$ given $Z$ is required for deriving the asymptotic results.

As an example of application of our procedure consider the case where $Y$ is a count variable. First, write the Poisson likelihood where the function $r(\cdot)$ is replaced by a kernel estimator and maximize this likelihood with respect to $\theta $\ to obtain a semiparametric PML estimator of $\theta _{0}$. Use this estimate and the variance formula (\ref{mom2}) to deduce a consistent moment estimator of $\sigma ^{2}.$ In a second step, estimate $\theta _{0}$ through a semiparametric Negative Binomial PML where $r$ is again replaced by a kernel estimator and the
variance parameter of the Negative Binomial is set equal to the estimate of $\sigma ^{2}.$ Finally, given the second step estimate of $\theta _{0}$, build a kernel estimator for the regression $r(\cdot)$. For simplicity, we use a Nadaraya-Watson estimator to estimate $r(\cdot)$. Other smoothers like local polynomials could be used at the expense of more intricate technical
arguments.

The occurrence of a nonparametric estimator in a pseudo-likelihood criterion requires a rule for the smoothing parameter. While the semiparametric index regression literature contains a large amount of contributions on how to estimate an index, there are much less results and practical solutions on the choice of the smoothing parameter. Even if the smoothing parameter does
not influence the asymptotic variance of a semiparametric estimator of $\theta_0$, in practice the estimate of $\theta_0$ and of the regression function may be sensitive to the choice of the smoothing parameter.

Another contribution of this paper is to propose an automatic and natural choice of the smoothing parameter used to define the semiparametric estimator. For this, we extend the approach introduced by H\"{a}rdle, Hall and Ichimura (1993) (see also Xia and Li (1999), Xia, Tong and Li (1999) and Delecroix, Hristache and Patilea (2006)). The idea is to maximize the pseudo-likelihood simultaneously in $\theta $ and the smoothing parameter, that is the bandwidth of the kernel estimator. The bandwidth is allowed to belong to a large range between $n^{-1/4}$ and $n^{-1/8}$. In some sense, this approach considers the bandwidth an auxiliary parameter for which the pseudo-likelihood may provide an estimate. Using a suitable decomposition of the pseudo-log-likelihood we show that such a joint maximization is asymptotically equivalent to separate maximization of a purely parametric (nonlinear) term with respect to $\theta $ and minimization of a weighted (mean-squared) cross-validation function with respect to the bandwidth. The weights of this cross-validation function are given by the second order derivatives of the pseudo-log-likelihood with respect to $r$. We show that the rate of our `optimal' bandwidth is $n^{-1/5}$, as expected for twice differentiable regression functions.

The paper is organized as follows. In section \ref{metoda} we introduce a class of semiparametric PML estimators based on linear exponential densities with nuisance parameter and we provide a natural bandwidth choice. Moreover, we present the general methodology used for the asymptotics. Section \ref{rezultat} contains the asymptotic results. A bound for the variance of our semiparametric PML estimators is also derived.
In section \ref{twostep} we use the semiparametric PML estimators to define a two-step procedure that can be applied in single-index regression models where an additional variance condition like (\ref{mom2}) is specified. Section \ref{simulsec} examines the finite-sample properties of our procedure via Monte Carlo simulations. We compare the performances of a two-step generalized
least-squares with those of a Negative Binomial PML in a Poisson single-index regression model with multiplicative unobserved heterogeneity. Even if the two procedures considered lead to asymptotically equivalent estimates, the latter procedure seems preferable in finite samples. An
application to real data on the frequency of recreational trips (see Cameron and Trivedi (2013), page 246) is also provided.  Section \ref{concl} concludes the paper. The technical proofs are postponed to the Appendix.

\section{Semiparametric PML with nuisance parameter}

\setcounter{equation}{0} \label{metoda}

Consider that the observations $\left( Y_{1},Z_{1}^{T}\right)^{T},...,\left( Y_{n}, Z_{n}^{T} \right) ^{T}$ are independent copies of the random vector $\left( Y,Z^{T}\right) ^{T}\in \mathbb{R} \times \mathbb{R}^{d}.$ Assume that there exists $\theta _{0}\in \mathbb{R}^{d}$, unique up to a scale normalization factor,{\Large \ }such that the single-index model (SIM)
condition
\begin{equation}
E\left(  Y\mid Z \right)  =E\left(  Y\mid Z^{T}\theta _{0}\right)  =r\left( Z^{T}\theta _{0};\theta _{0}\right)   \label{sim}
\end{equation}
holds. In this paper, we focus on single-index models where the conditional second order moment of $Y$ given $Z$ is a known function of $E\left[ Y\mid Z \right] $ and of a nuisance parameter. To be more precise, in the model we consider,
\begin{equation} \label{var1}
Var\left( Y\mid Z\right) =g\left( E\left( Y\mid Z\right) ,\alpha _{0}\right) =g\left( r\left( Z^{T}\theta _{0};\theta _{0}\right) ,\alpha _{0}\right) ,
\end{equation}
for some real value $\alpha _{0}$. The function $g\left( \cdot ,\cdot \right) $ is known and, for each $r,$ the map $\alpha \rightarrow g\left( r,\alpha \right) $ is one-to-one. Our framework is slightly more general that the one considered by Cui, H\"{a}rdle and Zhu (2011) where the conditional variance of $Y$ given $Z$ is a given function of the conditional mean of $Y$ given $Z$ multiplied by  an unknown constant.

To estimate the parameter of interest $\theta _{0}$ in a model like (\ref{sim})-(\ref{var1}), we propose a semiparametric PML procedure based on linear exponential families with nuisance parameter. The density used to build the pseudo-likelihood is taken with mean and variance equal to $r$ and $g(r,\alpha )$, respectively. In this section we suppose that an estimator of the nuisance parameter is given. In section \ref{twostep} we show how to build such an estimator using
a preliminary estimate of $\theta _{0}$ and condition (\ref{var1}).

\subsection{Linear exponential families with nuisance parameter}

Gouri\'{e}roux, Monfort and Trognon (1984a) introduced a class of densities, with respect to a given measure $\mu$, called linear exponential family with nuisance parameter (LEFN) and defined as
\begin{equation*}
l\left( y\mid r,\alpha \right) =\exp \left[ B\left( r,\alpha \right) +C\left( r,\alpha \right) y+D\left( y,\alpha \right) \right] ,
\end{equation*}
where $\alpha $ is the nuisance parameter. Since the dominating measure $\mu$ need not be Lebesgue measure, the law defined by $l$ is not necessarily continuous. The functions $B\left( \cdot ,\cdot \right) $ and $C\left( \cdot ,\cdot \right) $ are such that the expectation of the corresponding law is $r $ while the variance is $\left[ \partial_{r} C\left( r,\alpha \right) \right] ^{-1}.$ ($\partial_{r}$ denotes the derivative with respect to the argument $r.$) Recall that for any given $\alpha ,$ the following identity holds:
\begin{equation*}
\partial _{r}B\left( r,\alpha \right) +\partial _{r}C\left( r,\alpha \right) r\equiv 0.
\end{equation*}
If $\alpha $ is fixed, a LEFN becomes a linear exponential family (LEF) of densities. Gouri\'{e}roux, Monfort and Trognon (1984a, 1984b) used LEFN densities to define a two-step PML procedure in nonlinear regression models where a specification of the conditional variance is given. Herein, we extend their approach to a semiparametric framework.

In the case of the SIM defined by equation (\ref{mom2}), the conditional variance is given by $g\left( r,\alpha \right) =r\left( 1+\alpha r\right) $ with $r$ and $\alpha >0.$ In this case take
\begin{equation*}
B\left( r,\alpha \right) =-\frac{1}{\alpha }\ln \left( 1+\alpha r\right) \quad \quad \text{and}\quad \quad C\left( r,\alpha \right) =\ln \frac{r}{1+\alpha r},
\end{equation*}
which define a Negative Binomial distribution of mean $r$ and variance $r\left( 1+\alpha r\right) $. Note that the limit case $\alpha =0$ corresponds to a Poisson distribution. As another example, consider $g\left( r,\alpha \right) =r^{2}/\alpha $ with $r$ and $\alpha >0.$ Now, take the LEFN density given by $B\left( r,\alpha \right) =-\alpha \ln r$ and $C\left( r,\alpha \right) =-\alpha /r,$ which is the density of a gamma law of mean $r $ and variance $r^{2}/\alpha $.

\subsection{The semiparametric estimator}

In order to define our semiparametric PML estimator in the presence of a nuisance parameter let us introduce some notation: given $\left\{ c_{n}\right\} ,$ a sequence of numbers growing slowly to infinity (e.g., $c_{n}=\ln n$), let
\begin{equation*}
\mathcal{H}_{n}=\left\{ h:\,\,c_{n}\,n^{-1/4}\leq h\leq c_{n}^{-1}n^{-1/8}\right\}
\end{equation*}
be the range from which the `optimal' bandwidth will be chosen. Define the set $\Theta _{n} =\left\{ \theta :\left\| \theta -\theta _{0}\right\| \leq d_{n}\right\} $, $n\geq 1,$ with $\left\{ d_{n}\right\} $ some sequence decreasing to zero.

Let $\alpha ^{\ast }$ be some real value of the nuisance parameter$.$ Typically, $\alpha ^{\ast }=\alpha _{0}$ if the conditional variance formula (\ref{var1}) is correctly specified.\ Otherwise, $\alpha ^{\ast }$ is some pseudo-true value of the nuisance parameter. Suppose that a sequence $\left\{ \widetilde{\alpha }_{n}\right\} $ such that $\widetilde{\alpha }_{n}\rightarrow \alpha ^{\ast }$, in probability, is given. Set\footnote{Herein, we focus on $\psi \left( y,r;\alpha \right) =\ln l\left( y\mid r,\alpha \right) $ where $l\left( y\mid r,\alpha \right)
=\exp \left[ B\left( r,\alpha \right) +C\left( r,\alpha \right) y+D\left( y,\alpha \right) \right] $ is a LEFN density. However, other functions $\psi \left( y,r;\alpha \right) $ having the required properties can be considered (see Appendix \ref{assu}).}
\begin{equation*}
\psi \left( y,r;\alpha \right) =\ln l\left( y\mid r,\alpha \right)
\end{equation*}%
with $l\left( y\mid r,\alpha \right) $ the LEFN density of expectation $r$ and nuisance parameter $\alpha $. Define the semiparametric PML estimator in the presence of a nuisance parameter and the optimal bandwidth as
\begin{equation}
\left( \widehat{\theta },\widehat{h}\right) =\ \underset{\theta \in \Theta _{n},\,h\in \mathcal{H}_{n}}{\arg \max }\ \frac{1}{n}\sum\limits_{i=1}^{n} \psi \left( Y_{i}, \, \hat{r}_{h}^{i}\left( Z_{i}^{T}\theta ;\theta \right) ;\, \widetilde{\alpha }_{n}\right) \,\tau _{n}(Z_{i}),  \label{deeff}
\end{equation}
where
\begin{equation*}
\hat{r}_{h}^{i}\left( t;\theta \right) =\frac{\frac{1}{n-1} \sum\limits_{j\neq i}Y_{j}\ K_{h}\left( t-Z_{j}^{T}\theta \right) \ }{\frac{1}{n-1}\sum\limits_{j\neq i}K_{h}\left( t-Z_{j}^{T}\theta \right) \ }=:\frac{\widehat{\gamma }_{h}^{i}\left( t;\theta \right) } {\widehat{f}_{h}^{i}\left( t;\theta \right) }
\end{equation*}
denotes the leave-one-out version of the Nadaraya-Watson estimator of the regression function
\begin{equation*}
r\left( t;\theta \right) =E\left( Y|Z^{T}\theta =t\right) =: \frac{\gamma \left( t;\theta \right) }{f\left( t;\theta \right) },
\end{equation*}
with $f\left( \cdot ;\theta \right) $ the density of $Z^{T}\theta .$ The function $K\left( \cdot \right) $ is a second order kernel function and $K_{h}\left( \cdot \right) $ stands for $K\left( \cdot /h\right) /h,$ where $h$ is the bandwidth. $\tau _{n}(\cdot )$ denotes a trimming function. If the sequence $\widetilde{\alpha }_{n}$ is constant or $\psi $ does not depend on $\alpha ,$ equation (\ref{deeff}) defines a semiparametric PML based on a LEF density.

A trimming is designed to keep the density estimator $\widehat{f}_{h}^{i}$ away from zero in computations and it is usually required for analyzing the asymptotic properties of the nonparametric regression estimator and of the `optimal' bandwidth. The practical purpose of a trimming recommends a data-driven device like $I_{\left\{ z:\,\widehat{f}_{h}^{i}\left(
z^{T}\theta ;\theta \right) \geq c\right\} }(\cdot )$, with some fixed $c>0$ . (Herein, $I_{A}\left( \cdot \right) $ denotes the indicator function of the set $A.$) However, to ensure consistency with such a trimming, one should require in addition that
\begin{equation*}
\theta _{0}=\arg \max_{\theta }E\left[ \psi \left( Y,r_{\theta }\left( Z^{T}\theta \right) \right) \,I_{\left\{ z:\,f\left( z^{T}\theta ;\theta \right) \geq c\right\} }(Z)\right] .
\end{equation*}
Meanwhile, a trimming like $I_{\left\{ z:\,f\left( z^{T}\theta _{0};\theta_{0}\right) \geq c\right\} }(\cdot )$ is easier to handle in theory. Here, we consider
\begin{equation}
\tau _{n}(\cdot )=I_{\left\{ z:\,\widehat{f}_{h_{n}}^{i}\left( z^{T}\theta _{n};\,\theta _{n}\right) \geq c\right\} }(\cdot )  \label{deeff2}
\end{equation}
with $\theta _{n}\in \Theta _{n},$ $n\geq 1$, a sequence with limit $\theta _{0}$ and $h_{n},$ $n\geq 1$, a sequence of preliminary bandwidths such that $ n^{\varepsilon }h_{n}\rightarrow 0$ and $n^{1/2-\varepsilon }h_{n}\rightarrow \infty $ for some $0<\varepsilon <1/2.$ The trimming procedure we propose represents an appealing compromise between the theory and the applications.\ On one hand, it is easy to implement. On the other hand, we show below that, in a certain sense, our trimming is asymptotically equivalent to the fixed trimming $I_{\left\{ z:\,f\left( z^{T}\theta _{0};\theta _{0}\right) \geq c\right\} }(\cdot )$ and this fact greatly simplifies the proofs. We prove this equivalence under two types of assumptions: either i) $Z$ is bounded and $\theta _{n}-\theta _{0}=o\left( 1\right) $, or ii) $E\left[ \exp \left( \lambda \left\| Z\right\| \right) \right] <\infty ,$ for some $\lambda >0,$ and $\theta _{n}-\theta _{0}=o\left( 1/\ln n\right) .$ To be more precise, define $A=\left\{ z:f\left( z^{T}\theta _{0};\theta _{0}\right) \geq c\right\} \subset \mathbb{R}^{d}$ and $A^{\delta }=\left\{ z:\left| f\left( z^{T}\theta _{0};\theta _{0}\right) -c\right| \leq \delta \right\} ,$ $\delta >0.$ By little algebra, for all $\theta \in \Theta _{n},$ $h$ and $i,$
\begin{equation*}
\left| I_{\left\{ z:\,\widehat{f}_{h}^{i}\left( z^{T}\theta ;\theta \right) \geq c\right\} }(Z_{i})-I_{A}(Z_{i})\right| \leq I_{A^{\delta }}(Z_{i})+I_{(\delta ,\infty )}(G_{n}),
\end{equation*}
where
\begin{equation*}
G_{n}=\max_{1\leq i\leq n}\,\sup_{\theta \in \Theta _{n},\,h}\,\left| \widehat{f}_{h}^{i}\left( Z_{i}^{T}\theta ;\theta \right) -f\left( Z_{i}^{T}\theta _{0};\theta _{0}\right) \right| .
\end{equation*}
Let
\begin{equation*}
\widehat{S}\left( \theta ,h;\widetilde{\alpha }_{n},\overline{A}\right) = \dfrac{1}{n}\sum\limits_{i=1}^{n}\psi \left( Y_{i},\hat{r}_{h}^{i}\left( Z_{i}^{T}\theta ;\theta \right) ;\widetilde{\alpha }_{n}\right) \ I_{\overline{A}}\left( Z_{i}\right) \
\end{equation*}
with $\overline{A}=A$ or $A^{\delta }.$ Without loss of generality, consider that $\psi \left( \cdot ,\cdot \,;\cdot \right) \leq 0.$ (Since $\psi $ is the logarithm of a LEFN density$,$ for any given $y$ and $\alpha $, the map $r\rightarrow \psi \left( y,r\,;\alpha \right) $ attains its maximum at $r=y;$ thus, up to a translation with a function depending only on $y$ and $\alpha ,
$ we may consider $\psi \leq 0.$) In this case we have
\begin{align}
	&\left| \frac{1}{n}\sum\limits_{i=1}^{n}\psi \left( Y_{i},\hat{r}_{h}^{i}\left( Z_{i}^{T}\theta ;\theta \right) ;\widetilde{\alpha }_{n}\right) I_{\left\{ z:\, \widehat{f}_{h_{n}}^{i}\left( z^{T}\theta_{n};\theta _{n}\right) \geq c\right\} }(Z_{i})-\widehat{S}\left( \theta ,h; 	\widetilde{\alpha }_{n},A\right) \right| \hspace*{1cm}  \label{ikj} \\
	& \hspace*{2.5cm} \leq -\widehat{S}\left( \theta ,h; \widetilde{\alpha }_{n},A^{\delta }\right) -\frac{I_{(\delta ,\infty )}(G_{n})}{n}\sum\limits_{i=1}^{n}\psi \left( Y_{i}, \hat{r}_{h}^{i}\left( Z_{i}^{T}\theta ;\theta \right) ;\widetilde{\alpha }_{n}\right) .  \notag
\end{align}
We show that $\widehat{S}\left( \theta ,h;\alpha ,A^{\delta }\right) =o_{P}(\widehat{S}\left( \theta ,h;\alpha ,A\right) ),$ uniformly over $\Theta_{n}\times \mathcal{H}_{n}$ and uniformly in $\alpha ,$ provided that $\delta \rightarrow 0$ and $P\left(  f\left( Z^{T}\theta _{0};\theta
_{0}\right) =c \right)  =0.$ On the other hand, we prove that $P\left( G_{n}>\delta \right) \rightarrow 0,$ provided that $\delta \rightarrow 0$ slowly enough and $h\rightarrow 0$ faster than $n^{\varepsilon }$ and slower than $n^{1/2-\varepsilon },\,$for some $0<\varepsilon <1/2$. (See Lemma \ref{idic} in the appendix; in that lemma we distinguish two types of assumptions depending on whether $Z$ is bounded or not.)

Deduce that $\left( \widehat{\theta },\widehat{h}\right) $ is asymptotically equivalent to the maximizer of $\widehat{S}\left( \theta ,h;\widetilde{\alpha }_{n},A\right) $ over $\Theta _{n}\times \mathcal{H}_{n}.$ Therefore, hereafter, we simply write $\widehat{S}\left( \theta ,h;\widetilde{\alpha }_{n}\right) $ instead of $\widehat{S}\left( \theta ,h; \widetilde{\alpha}_{n}, A\right) $ and we consider
\begin{equation} \label{infeasi}
\left( \widehat{\theta },\widehat{h}\right) =\ \underset{\theta \in \Theta _{n},\,h\in \mathcal{H}_{n}} {\arg \max }\widehat{S}\left( \theta ,h;\widetilde{\alpha }_{n}\right) .
\end{equation}

\subsection{Methodology}

The semiparametric pseudo-log-likelihood $\widehat{S}\left( \theta ,h; \widetilde{\alpha }_{n}\right) $ can be split into a purely parametric (nonlinear) part $\widetilde{S}\left( \theta ;\widetilde{\alpha }_{n}\right) $, a purely nonparametric one $T(h;\alpha ^{\ast })$ and a reminder term $R(\theta ,h;\widetilde{\alpha }_{n})$, where
\begin{align}
\widetilde{S}\left( \theta ;\widetilde{\alpha }_{n}\right) &=\dfrac{1}{n} \sum\limits_{i=1}^{n}\left[ \psi \left( Y_{i},r\left( Z_{i}^{T}\theta ;\theta \right) ;\widetilde{\alpha }_{n}\right) -\psi \left( Y_{i},r\left( Z_{i}^{T}\theta _{0};\theta _{0}\right) ;\alpha ^{\ast }\right) \right] I_{A}\left( Z_{i}\right) ,  \label{deco} \\
T\left( h;\alpha ^{\ast }\right) &=\dfrac{1}{n}\sum\limits_{i=1}^{n}\psi \left( Y_{i},\hat{r}_{h}^{i}\left( Z_{i}^{T}\theta _{0};\theta _{0}\right) ;\alpha ^{\ast }\right) I_{A}\left( Z_{i}\right) ,  \notag \\
R\left( \theta ,h;\widetilde{\alpha }_{n}\right) &=\dfrac{1}{n} \sum\limits_{i=1}^{n}\left[ \psi \left( Y_{i},\hat{r}_{h}^{i}\left( Z_{i}^{T}\theta ;\theta \right) ;\widetilde{\alpha }_{n}\right) -\psi \left( Y_{i},r\left( Z_{i}^{T}\theta ;\theta \right) ;\widetilde{\alpha } _{n}\right) \right] I_{A}\left( Z_{i}\right)  \notag \\
& \hspace*{0.5cm} -\dfrac{1}{n}\sum\limits_{i=1}^{n}\left[ \psi \left( Y_{i},\hat{r} _{h}^{i}\left( Z_{i}^{T}\theta _{0};\theta _{0}\right) ;\alpha ^{\ast}\right) -\psi \left( Y_{i},r\left( Z_{i}^{T}\theta _{0};\theta _{0}\right) ;\alpha ^{\ast }\right) \right] I_{A}\left( Z_{i}\right)  \notag
\end{align}
(see H\"{a}rdle, Hall and Ichimura (1993) for a slightly different splitting). Given this decomposition, the simultaneous optimization of $\widehat{S}\left( \theta ,h;\widetilde{\alpha }_{n}\right) $ is asymptotically equivalent to separately maximizing $\widetilde{S}\left( \theta ;\widetilde{\alpha }_{n}\right) $ with respect to $\theta $ and $T\left( h;\alpha ^{\ast }\right) $ with respect to $h$, provided that $R\left( \theta ,h;\widetilde{\alpha }_{n}\right) $ is sufficiently small.\

A key ingredient for proving that $R\left( \theta ,h;\widetilde{\alpha } _{n}\right) $ is negligible with respect to $\widetilde{S}\left( \theta ; \widetilde{\alpha }_{n}\right) $ and $T\left( h;\alpha ^{\ast }\right) ,$ uniformly in $\left( \theta ,h\right) \in \Theta _{n}\times \mathcal{H}_{n}$ and for any $\left\{ \widetilde{\alpha }_{n}\right\} ,$ is represented by the orthogonality conditions
\begin{equation}
E\left[ \partial _{2}\psi \left( Y,\;r\left( Z^{T}\theta _{0};\theta _{0}\right) ;\alpha \right) \mid Z\right] =0  \label{orth1}
\end{equation}
and
\begin{equation} \label{orth2}
E\left[ \partial _{\theta }\partial _{2}\psi \left( Y,\;r\left( Z^{T}\theta _{0};\theta _{0}\right) ;\alpha \right) \mid Z^{T}\theta _{0}\right] =0,
\end{equation}
that must hold for any $\alpha ,$ where $\partial _{2}$ denotes the derivative with respect to the second argument of $\psi \left( \cdot ,\cdot ;\cdot \right) $ and $\partial _{\theta }$ is the derivative with respect to all occurrences of $\theta ,$ that is given $y$, $z$ and $\alpha ,$%
\begin{equation*}
\partial _{\theta }\partial _{2}\psi \left( y,\;r\left( z^{T}\theta _{0};\theta _{0}\right) ;\alpha \right) =\frac{\partial }{\partial \theta } \left. \partial _{2}\psi \left( y,\;r\left( z^{T}\theta ;\theta \right) ;\alpha \right) \right| _{\theta =\theta _{0}}
\end{equation*}
(see also Sherman (1994b) and Delecroix, Hristache and Patilea (2006) for similar conditions). If
\begin{equation*}
\psi \left( y,r;\alpha \right) =\ln l\left( y\mid r,\alpha \right) =B\left( r,\alpha \right) +C\left( r,\alpha \right) y+D\left( y,\alpha \right) ,
\end{equation*}
with $\partial _{r}B\left( r,\alpha \right) +\partial _{r}C\left( r,\alpha \right) r\equiv 0,$ then $\partial _{2}\psi \left( y,r;\alpha \right) =\partial _{r}C\left( r,\alpha \right) \left( y-r\right) $ and thus (\ref{orth1}) is a consequence of the SIM condition (\ref{sim}). To check the
second orthogonality condition note that
\begin{equation*}
E\left[ \partial _{22}^{2}\psi \left( Y,\;r\left( Z^{T}\theta _{0};\theta _{0}\right) ;\alpha \right) \mid Z\right] =E\left[ \partial _{22}^{2}\psi \left( Y,\;r\left( Z^{T}\theta _{0};\theta _{0}\right) ;\alpha \right) \mid Z^{T}\theta _{0}\right]
\end{equation*}
and
\begin{equation*}
E\left[ \partial _{\theta }r\left( Z^{T}\theta _{0};\theta _{0}\right) \mid Z^{T}\theta _{0}\right] =E\left[ r^{\prime }\left( Z^{T}\theta _{0};\theta _{0}\right) \left( Z-E\left[ Z\mid Z^{T}\theta _{0}\right] \right) \mid Z^{T}\theta _{0}\right] ,
\end{equation*}
where $r^{\prime }(\cdot ;\theta _{0})$ is the derivative of $r(\cdot ;\theta _{0}).$ The last identity is always true under the SIM condition (e.g., Newey~(1994), page 1358)$.$ Let us point out that conditions (\ref{orth1})-(\ref{orth2}) hold even if the variance condition (\ref{var1}) is
misspecified.

Since $R\left( \theta ,h;\widetilde{\alpha }_{n}\right) $ is negligible with respect to $\widetilde{S} \left( \theta ;\widetilde{\alpha }_{n}\right) $ and $T\left( h;\alpha ^{\ast }\right) $ does not contain the parameter of interest, the asymptotic distribution of $\widehat{\theta }$ will be obtained by standard arguments used for $M-$estimators in the presence of nuisance parameters applied to the objective function $\widetilde{S}\left( \theta ;\widetilde{\alpha }_{n}\right) $. We deduce that $\widehat{\theta }$ behaves as follows: i) if the SIM condition (\ref{sim}) holds and $\widetilde{\alpha }_{n}-\alpha ^{\ast }=O_{P}\left( 1\right) ,$ for some $\alpha ^{\ast },$ then $\widehat{\theta }$ is asymptotically normal; ii) if
SIM condition holds, the conditional variance (\ref{var1}) is correctly specified and $\widetilde{\alpha }_{n}-\alpha _{0}=O_{P}\left( 1\right) ,$ then $\widehat{\theta }$ is asymptotically normal and it has the lowest variance among the semiparametric PML estimators based on LEF densities. In any case, the asymptotic distribution of $\sqrt{n}(\widehat{\theta }-\theta
_{0})$ does not depend on the choice of $\widetilde{\alpha }_{n}.$ Let us point out that in our framework we only impose $\widetilde{\alpha }_{n}$ convergent in probability without asking a rate of convergence $O_{P}\left( 1/\sqrt{n}\right) ,$ as it is usually supposed for $M-$estimation in the presence of nuisance parameters. This because the usual orthogonality condition $E \left[ \partial _{\alpha }\partial _{\theta }\psi \left( Y,r\left( Z^{T}\theta _{0};\theta _{0}\right) ;\alpha \right) \right] =0$ is true for any $\alpha ,$ provided that $\psi \left( y,r;\alpha \right) =\ln l\left( y\mid r,\alpha \right) $ with $l\left( y\mid r,\alpha \right) $ a LEFN density$.$ Indeed, we have
\begin{align*}
& \hskip-0.8cm E\left[ \partial _{\alpha }\partial _{\theta }\psi \left( Y,r\left( Z^{T}\theta _{0};\theta _{0}\right) ;\alpha \right) \right] \\
&=E\left[ \partial _{\alpha }\partial _{r}\psi \left( Y,r\left( Z^{T}\theta_{0};\theta _{0}\right) ;\alpha \right) \partial _{\theta }r\left( Z^{T}\theta _{0};\theta _{0}\right) \right] \\
&=E\left[ E\left\{ \partial _{\alpha } \partial _{r} B\left( r\left( Z^{T}\theta _{0};\theta _{0}\right) ;\alpha \right) + \partial _{\alpha } \partial _{r} C\left( r\left( Z^{T}\theta _{0};\theta _{0}\right) ;\alpha \right) Y\mid Z\right\} \partial _{\theta }r\left( Z^{T}\theta _{0};\theta _{0}\right) \right] \\
&=0
\end{align*}
because $E\left(  Y\mid Z\right)  =r\left( Z^{T}\theta _{0};\theta _{0}\right) $ and $\partial _{\alpha } \partial _{r} B\left( r,\alpha \right) + \partial _{\alpha } \partial _{r} C\left( r,\alpha \right) r\equiv 0,$ for any $\alpha . $

For the bandwidth $\widehat{h}$ we obtain an asymptotic equivalence with a theoretical\ `optimal' bandwidth minimizing $-T\left( h;\alpha ^{\ast}\right) ,$ that is we prove that the ratio of the two bandwidths converges to one, in probability. Remark that $-T\left( h;\alpha ^{\ast }\right) $ is a kind of $\psi -$CV (cross validation) function.\ It can be shown that, up to constant additive terms, $-T\left( h;\alpha ^{\ast }\right) $ is asymptotically equivalent to a weighted (mean-squared) CV function. When $\psi \left( y,r;\alpha \right) =-\left( y-r\right) ^{2},$ the function $-T\left( h;\alpha ^{\ast }\right) $ is the usual CV function that one would use for choosing the bandwidth for the Nadaraya-Watson estimator of $E\left( Y\mid Z^{T}\theta _{0}\right) $. By extension of classical results for nonparametric regression, it can be proved that the rate of the theoretical `optimal' bandwidth minimizing $-T\left( h;\alpha ^{\ast }\right) $ is $%
n^{-1/5}$ (see Lemma \ref{mise} in Appendix \ref{appconv}; see also H\"{a}rdle, Hall and Ichimura (1993) for the case $\psi \left( y,r;\alpha \right) =-\left( y-r\right) ^{2}$)$.$ Deduce that $\widehat{h}$ is also of order $n^{-1/5}.$

\quad

\subsection{Extensions}

\label{relax}

Given the model conditions (\ref{sim})-(\ref{var1}), the idea is to choose a LEFN density with mean $r$ and variance $g(r,\alpha )$ and to construct a semiparametric PML estimator given a preliminary estimate of the nuisance parameter $\alpha _{0}$. However, it may happen that no such LEFN density exists or that one prefers another type of LEFN densities. Then, the idea is
to reparametrize the conditional variance of $Y$ given $Z$. More precisely, we may consider
\begin{equation*}
l\left( y\mid r,\eta \right) =\exp \left[ B\left( r,\eta \right) +C\left( r,\eta \right) y+D\left( y,\eta \right) \right] ,
\end{equation*}
where $\eta $ stands for the nuisance parameter. Let $\Sigma =\Sigma (r,\eta )$ denote the variance of the law given by this density. Assume that for any given $r$, the map $\eta \rightarrow \Sigma (r,\eta )$ is one-to-one. In this case, in order to provide a LEFN density with variance $g(r,\alpha )$ it suffices to consider $l\left( y\mid r,\eta \right) $ with $\eta =\Sigma
^{-1}(r,g(r,\alpha ))$. For instance, if $g(r,\alpha)=r(1+\alpha r^{2})$, one may use a Negative Binomial density of mean $r$ and nuisance parameter $\alpha r$. Another solution is to consider a normal density of mean $r$ where the variance equal to $r(1+\alpha r^{2})$ plays the role of the
nuisance parameter. In this case, given an estimate of $r(1+\alpha r^{2})$, our semiparametric PML becomes a semiparametric generalized least-squares (GLS) procedure. Note that this example of function $g(r, \alpha)$ leads us to the situation where the nuisance parameter is replaced by a `nuisance' function of $r$ and some additional parameters.

At the expense of more complicated writings, our methodology can be extended to take into account the case of a `nuisance' function. More precisely, consider a more general pseudo-log-likelihood function $\psi \left( y,r;\Psi (r,g(r,\alpha ))\right) $ where $\Psi (\cdot ,\cdot )$ is a given real-valued function and $\alpha $ is the nuisance parameter. See also Gouri\'{e}roux, Monfort and Trognon (1984a). To define $(\,\widehat{\theta }, \widehat{h}\,),$ one replaces $\widetilde{\alpha }_{n}$ by $\Psi (\widehat{r}_{h_{n}}(Z_{i}^{T}\theta _{n};\theta _{n});\widetilde{\alpha}_{n})$ in equation (\ref{deeff}), where $(\theta _{n},\widetilde{\alpha } _{n})\rightarrow (\theta_{0},\alpha
^{\ast })$, in probability, for some $ \alpha ^{\ast },$ and $\widehat{r}_{h_{n}}(\cdot ;\theta _{n})$ is a Nadaraya-Watson estimator of the regression $r(\cdot ;\theta _{n}).$ The same type
of decomposition of the pseudo-log-likelihood criterion into a purely parametric part function of $\theta $
\begin{multline*}
\dfrac{1}{n}\sum\limits_{i=1}^{n}\left[ \psi \left(  Y_{i},r\left( Z_{i}^{T}\theta ;\theta \right) ;\Psi \left(  r\left( Z_{i}^{T}\theta_{n};\theta _{n}\right) ,g\left( r(Z_{i}^{T}\theta_{n},\theta _{n}), \widetilde{\alpha }_{n}\right) \right) \right)  \right.  \\
\left. -\psi \left(  Y_{i},r\left( Z_{i}^{T}\theta _{0};\theta _{0}\right) ;\Psi \left(  r\left( Z_{i}^{T}\theta _{0};\theta _{0}\right) ,g\left( r(Z_{i}^{T}\theta _{0},\theta _{0}),\alpha ^{\ast }\right) \right) \right)  \right] I_{A}\left( Z_{i}\right) ,
\end{multline*}
a purely nonparametric part function of $h$
\begin{equation*}
T\left( h;\alpha ^{\ast }\right) =\dfrac{1}{n}\sum\limits_{i=1}^{n}\psi \left(  Y_{i}, \hat{r}_{h}^{i}\left( Z_{i}^{T}\theta _{0};\theta _{0}\right) ;\Psi \left(  r\left( Z_{i}^{T}\theta _{0};\theta _{0}\right) ,g\left( r(Z_{i}^{T}\theta _{0},\theta _{0}),\alpha ^{\ast } \right) \right) \right) I_{A}\left( Z_{i}\right)
\end{equation*}
and a negligible reminder function of $\theta $ and $h$ can be used. For brevity, the details of this more general case are omitted. However, we sketch a quick argument that applies for the semiparametric GLS.\footnote{This semiparametric generalized least-squares procedure is a particular case for Picone and Butler (2000). However, they do not provide a bandwidth rule.}
Consider the semiparametric GLS criterion
\begin{equation*}
\widehat{S}\left( \theta ,h;\theta _{n},\widetilde{\alpha }_{n},h_{n}\right) =-\,\dfrac{1}{n} \sum\limits_{i=1}^{n}g\left( \widehat{r}_{h_{n}}(Z_{i}^{T} \theta _{n};\theta_{n}); \widetilde{\alpha }_{n}\right) ^{-1}\left[ Y_{i}- \hat{r}_{h}^{i}\left( Z_{i}^{T}\theta ;\theta \right) \right] ^{2}I_{A}\left( Z_{i}\right)
\end{equation*}
with $(\theta _{n},\widetilde{\alpha }_{n})\rightarrow (\theta _{0},\alpha^{\ast })$, in probability, and $h_{n},$ $n\geq 1,$ a sequence of bandwidths.\ Assume that
\begin{equation}
\max_{1\leq i\leq n}\left| g\left( \widehat{r}_{h_{n}}(Z_{i}^{T}\theta _{n};\theta _{n});\widetilde{\alpha }_{n}\right) -g\left( r(Z_{i}^{T}\theta _{0};\theta _{0});\alpha ^{\ast }\right) \right| I_{A}\left( Z_{i}\right) =o_{P}\left( 1\right)   \label{uunif}
\end{equation}
and $g\left( r(z^{T}\theta _{0};\theta _{0});\alpha ^{\ast }\right) I_{A}\left( z\right) $ stays away from zero.\ Then the GLS criterion $\widehat{S} \left( \theta ,h;\theta _{n},\widetilde{\alpha }_{n},h_{n}\right) $ is asymptotically equivalent to the infeasible GLS criterion
\begin{equation*}
-\,\dfrac{1}{n}\sum\limits_{i=1}^{n}\left[ Y_{i}-\hat{r}_{h}^{i}\left( Z_{i}^{T}\theta ;\theta \right) \right] ^{2}g\left( r(Z_{i}^{T}\theta _{0};\theta _{0});\alpha ^{\ast }\right)^{-1}I_{A} \left( Z_{i}\right) ,
\end{equation*}
that is we can decompose the two criteria in such way that, up to negligible reminders, they have exactly the same purely parametric and purely nonparametric parts. Finally, we apply the methodology\footnote{Notice that the trimming function $z\rightarrow I_{A}\left( z\right) $ with $A=\left\{ z:f\left( z^{T}\theta _{0};\theta _{0}\right) \geq c\right\} $ can be written as a function of $z^{T}\theta _{0}.$ In view of our proofs, it becomes obvious that the methodology described in the previous subsection remains valid if $I_{A}\left( Z_{i}\right) $ is multiplied by a function depending only on $Z_{i}^{T}\theta _{0}.$} described in the previous subsection with $\psi \left( y,r;\alpha \right) =-\left( y-r\right) ^{2}$ and the trimming $I_{A}\left( Z_{i}\right) $ multiplied by $g\left( r(Z_{i}^{T}\theta _{0};\theta _{0});\alpha ^{\ast }\right) ^{-1}.\;$In order to ensure condition (\ref{uunif}), it suffices to suppose that the map $\left( r,\alpha \right) \rightarrow g\left( r;\alpha \right) $ satisfies a Lipschitz condition and that $h_{n}$ is such that
\begin{equation*}
\max_{1\leq i\leq n}\left| \widehat{r}_{h_{n}}(Z_{i}^{T}\theta _{0};\theta _{0})-r(Z_{i}^{T}\theta _{0};\theta _{0})\right| I_{A}\left( Z_{i}\right) =o_{P}\left( 1\right)
\end{equation*}%
and $\max_{1\leq i\leq n}|\partial _{\theta }\widehat{r}_{h_{n}}\left( Z_{i}^{T}\theta ;\theta \right) |I_{A}\left( Z_{i}\right) $ is bounded in probability, uniformly with respect to $\theta $ in $o_{P}\left( 1\right) $ neighborhoods of $\theta _{0}.$ For instance, a bandwidth of order $n^{-1/5}$ satisfies these conditions (see Andrews (1995); see also Delecroix, Hristache and Patilea (2006)).

Other possible extensions of the framework we consider is to allow a multi-index regression and/or multivariate dependent variables. For instance, the SIM condition can be replaced by the multi-index condition
\begin{equation*}
E \left(  Y\mid Z \right)  =E \left(  Y\mid Z^{T}\theta _{0}^{1},...,Z^{T}\theta _{0}^{p}\right)
\end{equation*}
with $p$ smaller than the dimension of $Z,$ while the second order moment condition remains $Var\left( Y\mid Z\right) =g\left( E\left( Y\mid Z\right) ,\alpha _{0}\right) .$ On the other hand, for multivariate dependent variables one may consider PML estimation based on the multivariate normal or multivariate generalizations of Poisson, Negative Binomial distributions
(Johnson, Kotz and Balakrishnan (1997)). The decomposition of the pseudo-log-likelihood in $\widetilde{S}$, $T$ and $R$ as above can still be used for these cases but the detailed analysis of these extensions will be considered elsewhere.

\section{Asymptotic results}

\label{rezultat}

\setcounter{equation}{0}

In this section we obtain the asymptotic distribution for $\widehat{\theta }$ and the corresponding estimator of the regression function $r\left( t;\theta \right) =E\left[ Y\mid Z^{T}\theta =t\right] $ as well as the asymptotic behavior of $\widehat{h}$, with $(\widehat{\theta },\widehat{h})$ defined in (\ref{deeff}). A consistent estimator for the asymptotic variance matrix of $\widehat{\theta }$ is proposed$.$ Moreover, a lower bound for the asymptotic variance matrix of $\widehat{\theta }$ is derived.

For the identifiability of the parameter of interest $\theta _{0}$, hereafter fix its first component, that is $\theta _{0}=(1,\widetilde{\theta }_{0}^{T})^{T},$ $\widetilde{\theta }_{0}\in \mathbb{R}^{d-1}.$ Therefore, we shall implicitly identify a vector $\theta =(1,\widetilde{\theta }
^{T})^{T}$ with its last $d-1$ components and redefine the symbol $\partial _{\theta }$ as being the vector of the first order partial derivatives with respect to the last $d-1$ components of $\theta .$

Let $v\left( t;\theta \right) =Var\left( Y\mid X\theta =t\right) .$ If the SIM assumption and variance condition (\ref{var1}) hold, then $v\left( Z^{T}\theta _{0};\theta _{0}\right) =g\left( r\left( Z^{T}\theta _{0};\theta _{0}\right) ,\alpha _{0}\right) .$ For a given $\theta ,$ let $r^{\prime }\left( \cdot ;\theta \right) $ and $r^{\prime \prime }\left( \cdot ;\theta \right) $ denote the first and second order derivatives of the function $r\left( \cdot ;\theta \right) .$ Similarly, $f^{\,\prime }\left( \cdot ;\theta \right) $ is the derivative of $f^{\,\prime }\left( \cdot ;\theta \right) .$ Define\footnote{Note that $\partial _{22}^{2}\psi (y,r)=\partial _{rr}^{2}C(r,\alpha )\left(y-r\right) -\partial _{r}C(r,\alpha ).$ Thus, $-\partial _{r}C$ can be
replaced by $\partial _{22}^{2}\psi $ in the definition of the constants $C_{1}$ and $C_{2}.$}
\begin{align} \label{a1a2}
C_{1} &=-\,\dfrac{K_{1}^{2}}{4}\ E\left\{ \underset{_{\quad}} {\frac{1}{2}\ \partial _{r}C\left( r\left( Z^{T}\theta _{0};\theta _{0}\right) ;\alpha ^{\ast }\right)} \ \right. \    \\
&\qquad \qquad \qquad \times \left. \left[ r^{\prime \prime }\left( Z^{T}\theta _{0};\theta _{0}\right) +\frac{2\ r^{\prime }\left( Z^{T}\theta _{0};\theta _{0}\right) \ f^{\,\prime }\left( Z^{T}\theta _{0};\theta_{0}\right) }{f\left( Z^{T}\theta _{0};\theta _{0}\right) }\right] ^{2}\
I_{A}\left( Z\right) \right\}  \notag \\
C_{2} &=-\,K_{2}\ E\left\{ \frac{1}{2}\ \partial _{r}C\left( r\left( Z^{T}\theta _{0};\theta _{0}\right) ;\alpha ^{\ast }\right) \,\,\frac{1}{f\left( Z^{T}\theta _{0};\theta _{0}\right) }\ v\left( Z^{T}\theta_{0};\theta _{0}\right) \ I_{A}\left( Z\right) \right\} ,  \notag
\end{align}
with $K_{1}=\int u^{2}K\left( u\right) du$, $K_{2}=\int K^{2}\left( u\right) du,$ and consider
\begin{equation*}
h_{n}^{opt}\ =\ \underset{h}{\arg \max }\ \left( C_{1}h^{4}+C_{2}n^{-1}h^{-1}\right) =\left( C_{2}/4C_{1}\right) ^{1/5}n^{-1/5}.
\end{equation*}%
Define the $\left( d-1\right) \times \left( d-1\right) $ matrices
\begin{equation*}
I=E\left\{ \left[ \partial _{r}C\left( r\left( Z^{T}\theta _{0};\theta _{0}\right) ;\alpha ^{\ast }\right) \right] ^{2}v\left( Z^{T}\theta_{0};\theta _{0}\right) \partial _{\theta }r\left( Z^{T} \theta _{0};\theta_{0}\right) \partial _{\theta }r\left( Z^{T}\theta _{0};\theta _{0}\right)
^{T}I_{A}\left( Z\right) \right\}
\end{equation*}

\begin{equation*}
J=E\left[ \partial _{r}C\left( r\left( Z^{T}\theta _{0};\theta _{0}\right) ;\alpha ^{\ast }\right) \partial _{\theta }r\left( Z^{T}\theta _{0};\theta_{0}\right) \partial _{\theta }r\left( Z^{T}\theta _{0};\theta _{0}\right) ^{T} I_{A}\left( Z\right) \right] .
\end{equation*}
Note that $I=J$ if the variance condition (\ref{var1}) holds and $\alpha^{\ast }=\alpha _{0}.$

Now, we deduce the asymptotic normality of the semiparametric PML $\widehat{\theta }$ estimator in the presence of a nuisance parameter.  Moreover, we obtain the rate of decay to zero of the \ `optimal' bandwidth $\widehat{h}.$ The proof of the following result is given in Appendix ref{proof}.

\quad

\begin{theor}
\label{param1} Suppose that the assumptions in Appendix \ref{assu} hold. Define the set $\Theta _{n}=\left\{ \theta :\left\| \theta -\theta_{0}\right\| \leq d_{n}\right\} $, $n\geq 1$, with $d_{n}\ln n\rightarrow 0$ and $\widetilde{\alpha }_{n},$ $n\geq 1,$ such that $\widetilde{\alpha }%
_{n}-\alpha ^{\ast }=o_{P}(1)$. Fix $c>0.$ If $(\widehat{\theta },\widehat{h})$ is defined as in (\ref{deeff})-(\ref{deeff2}), then $\widehat{h}/h_{n}^{opt}\rightarrow 1,$ in probability, and
\begin{equation*}
\sqrt{n}\left( \widehat{\theta }-\theta _{0}\right) \overset{\mathcal{D}} {\longrightarrow} \mathcal{N}\left( 0,J^{-1}IJ^{-1}\right) .
\end{equation*}
If $Z$ is bounded, the same conclusion remains true for any sequence $d_{n}\rightarrow 0$.
\end{theor}

\qquad

In applications $J^{-1}IJ^{-1}$ is unknown and therefore it has to be consistently estimated. To this end, we propose an usual sandwich estimator of the asymptotic variance $J^{-1}IJ^{-1}$ (e.g., Ichimura (1993)). Let $\widehat{f}_{h}\left( \cdot ;\,\theta \right) $ denote the kernel estimator
for the density of $Z^{T}\theta .$ Define
\begin{multline*}
I_{n}=\frac{1}{n}\sum_{i=1}^{n}\left[ \partial _{r}C\left( \widehat{r}_{\widehat{h}}\left( Z_{i}^{T}\widehat{\theta };\widehat{\theta }\right) ; \widetilde{\alpha }_{n}\right) \right] ^{2}\left[ Y_{i}-\widehat{r}_{\widehat{h}}\left( Z_{i}^{T}\widehat{\theta };\widehat{\theta }\right) \right] ^{2} \\
\times \partial _{\theta }\widehat{r}_{\widehat{h}}\left( Z_{i}^{T}\widehat{\theta }; \widehat{\theta }\right) \partial _{\theta }\widehat{r}_{\widehat{h} }\left( Z_{i}^{T} \widehat{\theta };\widehat{\theta }\right) ^{T}I_{\left\{z:\,\widehat{f}_{\widehat{h}}\left( z^{T}\widehat{\theta };\,\widehat{\theta }\right) \geq c\right\} }(Z_{i})
\end{multline*}
\begin{equation*}
J_{n}=\frac{1}{n}\sum_{i=1}^{n}\partial _{r}C\left( \widehat{r}_{\widehat{h} }\left( Z_{i}^{T} \widehat{\theta };\widehat{\theta }\right) ;\widetilde{\alpha }_{n}\right) \partial _{\theta }\widehat{r}_{\widehat{h}}\left( Z_{i}^{T}\widehat{\theta };\widehat{\theta }\right) \partial _{\theta } \widehat{r}_{\widehat{h}}\left( Z_{i}^{T}\widehat{\theta };\widehat{\theta } \right) ^{T}I_{\left\{ z:\,\widehat{f}_{\widehat{h}}\left( z^{T}\widehat{\theta };\,\widehat{\theta }\right) \geq c\right\} }(Z_{i}).
\end{equation*}

\quad

\begin{proposition}
\label{estasv} Suppose that the conditions of Theorem \ref{param1}\ hold. Then, $J_{n}^{-1}I_{n}J_{n}^{-1}\rightarrow J^{-1}IJ^{-1},$ in probability.
\end{proposition}

\quad

\begin{proof}
The arguments are quite standard (e.g., Ichimura (1993), section 7). On one hand, the convergence in probability of $\widehat{\theta }$ and $\widetilde{\alpha }_{n}$ and, on the other hand, the convergence in probability of $\widehat{r}_{\widehat{h}}\left( z^{T}\theta ;\theta \right) $ and $\partial_{\theta }\widehat{r}_{\widehat{h}}\left( z^{T}\theta ;\theta \right) ,$ uniformly over $\theta $ in neighborhoods shrinking to $\theta _{0}$ and uniformly over $z\in A$ (e.g., Andrews (1995), Delecroix, Hristache and Patilea (2006)) imply $I_{n}\rightarrow I$ and $J_{n}\rightarrow J,$ in probability.
\end{proof}

\quad

Theorem \ref{param1} shows, in particular, that $\widehat{\theta }$ is asymptotically equivalent to the semiparametric PML based on the LEF pseudo-log-likelihood $\psi \left( y,r;\alpha ^{\ast }\right) =\ln f\left(y,r\mid \alpha ^{\ast }\right) .$ As in the parametric case, we can deduce a
lower bound for the asymptotic variance $J^{-1}IJ^{-1}$ with respect to semiparametric PML based on LEF densities. This bound is achieved by $\widehat{\theta }$ if the SIM assumption and the variance condition (\ref{var1}) hold and $\alpha ^{\ast }=\alpha _{0}.$ The proof of the following
proposition is identical to the proof of Property 5 of Gouri\'{e}roux, Monfort and Trognon (1984a, page 687) and thus it will be skipped.

\quad

\begin{proposition}
\label{bound} The set of asymptotic variance matrices of the semiparametric PML estimators based on linear exponential families has a lower bound equal to $\mathcal{K}$, where
\begin{equation*}
\mathcal{K}^{-1}=E\left\{ \left[ v\left( Z^{T}\theta _{0};\theta _{0}\right) \right] ^{-1}\partial _{\theta }r\left( Z^{T}\theta _{0};\theta _{0}\right) \partial _{\theta }r\left( Z^{T}\theta_{0}; \theta _{0}\right) ^{T}I_{A}\left( Z\right) \right\} .
\end{equation*}
\end{proposition}

\quad

Concerning the nonparametric part, we have the following result on theasymptotic distribution of the nonparametric estimator of the regression. The proof is omitted (see H\"{a}rdle and Stoker (1989)).

\begin{proposition}
\label{nonpar} Assume that the conditions of Theorem \ref{param1} are fulfilled. Then, for any $t$ such that $f\left( t;\theta _{0}\right) >0,$
\begin{equation*}
\sqrt{n\widehat{h}}\left( \widehat{r}_{\widehat{h}}\left( t;\widehat{\theta } \right) -r\left( t;\theta _{0}\right) -\widehat{h}^{2}\beta \left( t\right) \right) \overset{\mathcal{D}} {\longrightarrow }N\left( 0,\,\,K_{2}v(t;\theta_{0})f\left( t;\theta _{0}\right) ^{-1}\right)
\end{equation*}
where $\beta \left( t\right) =\left( K_{1}/2\right) \left[ r^{\prime \prime }\left( t;\theta _{0}\right) +2r^{\prime }\left( t;\theta _{0}\right) f^{\,\prime }\left( t;\theta _{0}\right) f\left( t;\theta _{0}\right) ^{-1} \right] .$
\end{proposition}

\quad

Note that, for any $z$ such that $f\left( z^{T}\theta _{0};\theta_{0}\right) >0,$
\begin{equation*}
\sqrt{n\widehat{h}}\left( \widehat{r}_{\widehat{h}}\left( z^{T}\widehat{\theta };\widehat{\theta }\right) -r\left( z^{T}\theta _{0};\theta_{0}\right) -\widehat{h}^{2}\beta \left( z^{T}\theta _{0}\right) \right) \overset{\mathcal{D}}{\rightarrow }N\left( 0,\,\,K_{2}v(z^{T}\theta_{0};\theta _{0}) f\left( z^{T}\theta _{0};\theta _{0}\right) ^{-1}\right) .
\end{equation*}
Indeed, use the results of Andrews (1995) to deduce that $\partial _{\theta } \widehat{r}_{\widehat{h}} \left( z^{T}\theta ;\theta \right) \rightarrow \partial _{\theta }r\left( z^{T}\theta ;\theta \right) ,$ in probability, uniformly over neighborhoods of $\theta _{0}$ where $f\left( z^{T}\theta ;\theta \right) $ stays away from zero. Therefore, we can write
\begin{multline*}
\widehat{r}_{\widehat{h}}\left( z^{T}\widehat{\theta };\widehat{\theta } \right) -r\left( z^{T}\theta _{0};\theta _{0}\right) =\widehat{r}_{\widehat{h}}\left( z^{T}\widehat{\theta };\widehat{\theta }\right) -\widehat{r}_{\widehat{h}}\left( z^{T}\theta _{0};\theta _{0}\right) +\widehat{r}_{\widehat{h}}\left( z^{T}\theta _{0};\theta _{0}\right) -r\left( z^{T}\theta
_{0};\theta _{0}\right) \\
=\partial _{\theta }\widehat{r}_{\widehat{h}}\left( z^{T}\theta _{0};\theta _{0}\right) \left( \widehat{\theta }-\theta _{0}\right) +o_{P}\left( \left\| \widehat{\theta }-\theta _{0}\right\| \right) +\widehat{r}_{\widehat{h} }\left( z^{T}\theta _{0};\theta _{0}\right) -r\left( z^{T}\theta _{0};\theta_{0}\right) \\
=O_{P}\left( \left\| \widehat{\theta }-\theta _{0}\right\| \right) +\widehat{r}_{\widehat{h}} \left( z^{T}\theta _{0};\theta _{0}\right) -r\left(z^{T}\theta _{0};\theta _{0}\right)
\end{multline*}
and obtain the asymptotic normality of $\widehat{r}_{\widehat{h}}(z^{T} \widehat{\theta };\widehat{\theta })$ as a consequence of the $\sqrt{n}-$consistency of $\widehat{\theta }$ and the asymptotic behavior of the Nadaraya-Watson estimator.

\section{Two-step semiparametric PML}

\label{twostep}

Here, we consider a two-step semiparametric PML procedure that can be applied in semiparametric single-index regression models when a conditional variance condition like
\begin{equation}
Var\left( Y\mid Z\right) =g\left( E\left( Y\mid Z\right) ,\alpha _{0}\right) =g\left( r\left( Z^{T}\theta _{0};\theta _{0}\right) ,\alpha _{0}\right) , \label{gmm}
\end{equation}
is specified. Assume that this conditional variance condition is correctly specified.\ At the end of this section we also discuss the misspecification case.

First, we have to build a sequence $\left\{ \theta _{n}\right\} $ with limit $\theta _{0}.$ Moreover, in the case of unbounded covariates, $\theta _{n}$ should approach $\theta _{0}$ faster than $1/\ln n.$\ For this purpose, we maximize with respect to $\theta $ a pseudo-likelihood based on a LEF density $l\left( y\mid r\right) $. We use a fixed trimming $I_{B}(\cdot )$ with $B$ a subset of $\mathbb{R}^{d}$ such that, for any $\theta $ and any $z\in B,$ we have $f\left( z^{T}\theta ;\theta \right) \geq c>0.$ To ensure consistency for such a PML estimator, we have to check that
\begin{equation}
\theta _{0}=\ \underset{\theta }{\arg \max }\,E\left[ \,\ln l\left( Y\mid r\left( Z^{T}\theta ;\theta \right) \right) \,I_{B}(Z)\right] ,  \label{idfg} \end{equation} and $\theta _{0}$ is unique with this property. Recall that the SIM condition specifies $\theta _{0}$ as the unique vector satisfying $E\left[ Y\mid Z\right] =E\left[ Y\mid Z^{T}\theta _{0}\right] .$ On the other hand, if $\ln l\left( y\mid r\right) =B\left( r\right) +C\left( r\right) y+D\left( y\right) ,$ then $B\left( m\right) +C\left( m\right) r\leq B\left( r\right) +C\left( r\right) r$ (cf. Property 4, Gouri\'{e}roux, Monfort and Trognon (1984a, page 684)). Deduce that for any $z$,
\begin{equation*}
\theta _{0}=\ \underset{\theta }{\arg \max }\,E\left[ \,\ln l\left( Y\mid r\left( z^{T}\theta ;\theta \right) \right) \right]
\end{equation*}
and $\theta _{0}$ is the unique maximizer. Hence, condition (\ref{idfg}) holds for any set $B.$ This leads us to the following definition of a preliminary estimator.

\quad

\textbf{STEP 1 (preliminary step).} Consider a sequence of bandwidths $h_{n}, $ $n\geq 1,$ such that $n^{\varepsilon }h_{n}\rightarrow 0$ and $n^{1/2-\varepsilon }h_{n}\rightarrow \infty $ for some $0<\varepsilon <1/2.$ Moreover, let $l\left( y\mid r\right) $ be a LEF density. Define%
\begin{equation*}
\theta _{n}=\ \underset{\theta }{\arg \max }\ \frac{1}{n}\sum \limits_{i=1}^{n}\ln l\left( Y_{i}\mid \hat{r}_{h_{n}}\left( Z_{i}^{T}\theta ;\theta \right) \right) \,I_{B}(Z_{i}).
\end{equation*}

\qquad

Delecroix, Hristache and Patilea (2006) showed that, under the regularity conditions required by Theorem \ref{param1}, we have $\theta _{n}-\theta _{0}=o_{P}\left( 1/\ln n\right) .$ Using the preliminary estimate $\theta _{n}$ and the variance condition (\ref{gmm}) we can build $\widetilde{\alpha }_{n},$ $n\geq 1,$ such that $\widetilde{\alpha }_{n}\rightarrow \alpha _{0}$%
, in probability (see the end of this section). Let $l\left( y\mid r,\alpha \right) $ denote a LEFN density with mean $r$ and variance $g\left( r,\alpha \right) .$ Consider $c_{n}\rightarrow \infty $ (e.g., $c_{n}=\ln n$), define $\mathcal{H}_{n}=\left\{ h:\,\,c_{n}\,n^{-1/4}\leq h\leq
c_{n}^{-1}n^{-1/8}\right\} $. Moreover, consider $\Theta _{n}=\left\{ \theta :\left\| \theta -\theta _{0}\right\| \leq d_{n}\right\} $, $n\geq 1$ with $\left\{ d_{n}\right\} $ as in Theorem \ref{param1}. Fix some small $c>0.$

\quad

\textbf{STEP 2.} Define
\begin{equation*}
\left( \widehat{\theta },\widehat{h}\right) =\ \underset{\theta \in \Theta _{n},\,h\in \mathcal{H}_{n}}{\arg \max }\ \frac{1}{n}\sum\limits_{i=1}^{n} \ln l \left(Y_{i} \mid \hat{r}_{h}^{i}\left( Z_{i}^{T}\theta ;\theta \right) ;\, \widetilde{\alpha }_{n}\right) \,I_{\left\{ z:\,\widehat{f} _{h_{n}}^{i}\left( z^{T}\theta _{n};\,\theta _{n}\right) \geq c\right\} }(Z_{i}),
\end{equation*}
with $\theta_n$ and $h_n$ from Step 1.

\quad

The following result is a direct consequence of Theorem \ref{param1}.

\begin{cor}
\label{2step} Suppose that the assumptions of Theorem \ref{param1} hold. If $\widehat{\theta }$ and $\widehat{h}$ are obtained as in Step 2 above, then
\begin{equation*}
\sqrt{n}\left( \widehat{\theta }-\theta _{0}\right) \overset{\mathcal{D}}{\longrightarrow }\mathcal{N}\left( 0,\mathcal{K}\right) ,
\end{equation*}
with
\begin{equation*}
\mathcal{K}^{-1}=E\left\{ \left[ v\left( Z^{T}\theta _{0};\theta _{0}\right) \right] ^{-1}\partial _{\theta }r\left( Z^{T}\theta _{0};\theta _{0}\right) \partial _{\theta }r\left( Z^{T}\theta _{0};\theta _{0}\right) ^{T}I_{A}\left( Z\right) \right\} .
\end{equation*}
Moreover,
\begin{equation*}
\frac{\widehat{h}}{\left( C_{2}/4C_{1}\right) ^{1/5}n^{-1/5}}\rightarrow 1,
\end{equation*}
in probability, where $C_{1}$ and $C_{2}$ are defined as in (\ref{a1a2}) with $\alpha ^{\ast } =\alpha _{0}.$
\end{cor}

\quad

\emph{Remark 1.} Let us point out that simultaneous optimization of the semiparametric criterion in Step 1 with respect to $\theta $, $\alpha $ and $h$ (or with respect to $\theta $ and $\alpha $ for a given $h$) is not recommended, even if the conditional variance $Var\left( Y\mid Z\right) $ is correctly specified. Indeed, if the true conditional distribution of $Y$ given $Z$ is not the one given by the LEFN density $l\left( y\mid r,\alpha \right) =\exp \psi \left( y,r;\alpha \right) ,$ joint optimization with respect to $\theta $ and $\alpha $ leads, in general, to an inconsistent
estimate of $\alpha _{0}.$ (This failure is well-known in the parametric case where $r$ is a known function; see comments of Cameron and Trivedi (2013), pages 84-85. In view of decomposition (\ref{deco}) we deduce that this fact also happens in the semiparametric framework where $r$ has to be estimated.) In this case the matrices $I$ and $J$ defined in section \ref{rezultat} are no longer equal and thus the asymptotic variance of the one-step semiparametric estimator of $\theta $ obtained by simultaneous maximization of the criterion in Step 1 with respect to $\theta $, $\alpha $ does not achieve the bound $\mathcal{K}$. However, when the SIM condition holds and the true conditional law of $Y$ is given by the LEFN density $l=\exp \psi ,$ our two-step estimator $\widehat{\theta }$ and the semiparametric MLE of $\theta _{0}$ obtained by simultaneous optimization with respect to $(\theta ,\alpha )$ are asymptotically equivalent.

\quad

\emph{Remark 2.} Note that if we ignore the efficiency loss due to trimming, $\mathcal{K}$ is equal to the efficiency bound in the semiparametric model defined \emph{only} by the single-index condition $E\left(  Y\mid Z\right)  =E \left(  Y\mid Z^{T}\theta _{0}\right)  $ when the variance condition (\ref{gmm}) holds. To see this, apply the bound of Newey and Stoker (1993) with the true variance given by (\ref{gmm}). Our two-stage estimator achieves this SIM efficiency bound (if the variance is well-specified). However, this SIM bound is not necessarily the two moment conditions model bound.\ The latter should take into account the variance condition (see Newey (1993), section 3.2, for a similar discussion in the parametric nonlinear regression framework). In other words our two-stage estimator has some optimality properties but it may not achieve the semiparametric efficiency bound of the two moment conditions model. The same remark applies for the two-stage semiparametric generalized least squares (GLS) procedure of H\"{a}rdle, Hall and Ichimura (1993) [see also Picone and Butler (2000)]. Achieving semiparametric efficiency when the first two moments are specified would be possible, for instance, by estimating higher orders conditional moments nonparametrically. However, in this case we face again the problem of the
curse of dimensionality that we tried to avoid by assuming the SIM condition.

\quad

To complete the definition of the two-step procedure above, we have to indicate how to build a consistent sequence $\left\{ \widetilde{\alpha }_{n}\right\} $.\ Such a sequence can be obtained from the moment condition (\ref{gmm}) after replacing $r\left( z^{T}\theta _{0};\theta _{0}\right) $ by a suitable estimator.\ This kind of procedure is commonly used in the semiparametric literature (e.g., Newey and McFadden (1994)).\ For simplicity, let us only consider the Negative Binomial case where, for any $z,$ we have
\begin{equation} \label{reg}
E\left[ \left( Y-E\left( Y\mid Z\right) \right) ^{2}\mid Z=z\right] =r\left( z^{T}\theta_{0}; \theta _{0}\right) \left[ 1+\alpha _{0}r\left( z^{T}\theta_{0};\theta _{0}\right) \right] .
\end{equation}
Consider a set $B\subset \mathbb{R}^{d}$ such that, for any $\theta $ and any $z\in B,$ we have $f\left( z^{T}\theta ;\theta \right) \geq c>0.$ We can write
\begin{equation*}
E\left\{ E\left[ \left( Y-r\left( Z^{T}\theta _{0};\theta _{0}\right) \right) ^{2}-r\left( Z^{T}\theta _{0};\theta _{0}\right) \mid Z\right] I_{B}\left( Z\right) \right\} =\alpha_{0} E\left\{ r\left( Z^{T}\theta_{0};\theta _{0}\right) ^{2}I_{B}\left( Z\right) \right\} .
\end{equation*}
Consequently, we may estimate\footnote{One can expect little influence of the choice of the bandwidth used to construct the $\widetilde{\alpha }_{n}$. This is indeed confirmed by the simulation experiments we report in section \ref{simulsec}. } $\alpha _{0}$ by
\begin{equation}
\widetilde{\alpha }_{n}=\frac{\frac{1}{n}\sum_{i=1}^{n}\left[ \left( Y_{i}- \widehat{r}_{h_{n}} \left( Z_{i}^{T}\theta _{n};\theta _{n}\right) \right)^{2}-\widehat{r}_{h_{n}}\left( Z_{i}^{T} \theta _{n};\theta _{n}\right) \right] I_{B}\left( Z_{i}\right)} {\frac{1}{n} \sum_{i=1}^{n}\widehat{r}_{h_{n}}\left( Z_{i}^{T}\theta _{n};\theta _{n}\right) ^{2}I_{B}\left(
Z_{i}\right) }  \label{nuis}
\end{equation}%
with $\theta _{n}$ and $h_{n}$ from Step 1 and $\widehat{r}_{h_{n}}$ the Nadaraya-Watson estimator with bandwidth $h_{n}$. Since $\theta _{n}\rightarrow \theta _{0},$ deduce that $\widetilde{\alpha }_{n}\rightarrow \alpha _{0},$ in probability (see also the arguments we used in subsection \ref{relax}).

Now, let us comment on what happens with our two-step procedure if the second order moment condition is misspecified, while the SIM condition still holds. In general, the sequence $\widetilde{\alpha }_{n}$ one may derive from the conditional variance condition and the preliminary estimate of $\theta _{0}$ is still convergent to some pseudo-true value $\alpha ^{\ast }$ of the nuisance parameter.\footnote{For instance, $\widetilde{\alpha }_{n}$ defined in (\ref{nuis}) is convergent in probability to
\begin{equation*}
\alpha ^{\ast }=\frac{E[\left( Y-r\left( Z^{T}\theta _{0};\theta _{0}\right) \right) ^{2} I_{B}(Z)]-E[r\left( Z^{T}\theta _{0};\theta _{0}\right) I_{B}(Z)]}{E[r(Z^{T}\theta _{0};\theta _{0})^{2} I_{B}(Z)]}.
\end{equation*}
To ensure that the limit of $\widetilde{\alpha }_{n}$ is positive, one may replace $\widetilde{\alpha }_{n}$ by $\max \left( \widetilde{\alpha }_{n},\rho \right) $ for some small but positive $\rho .$} Then, the behavior of $\left( \widehat{\theta },\widehat{h}\right) $ yielded by Step 2 is described by Theorem \ref{param1}, that is $\widehat{\theta }$ is still $%
\sqrt{n}-$asymptotically normal and $\widehat{h}$ is still of order $n^{-1/5}.$

Finally, if the SIM condition does not hold, then $\widehat{\theta }$ estimates a kind of first projection-pursuit direction. In this case, our procedure provides an alternative to minimum average (conditional) variance estimation (MAVE) procedure of Xia \emph{et al}. (2002). The novelty would be that the first projection direction is defined through a more flexible PML function than the usual least-squares criterion. This case will be analyzed elsewhere.

\section{Empirical evidence}

In our empirical section we consider the case of a count response variable $Y$. A benchmark model for studying event counts is the Poisson regression model. Different variants of the Poisson regression have been used in applications on the number of patents applied for and received by firms, bank failures, worker absenteeism, airline or car accidents, doctor visits, \emph{etc.}
Cameron and Trivedi (2013) provide an overview of the applications of Poisson regression. In the basic setup, the regression function is log-linear. An additional unobserved multiplicative random error term in the conditional mean function is usually used to account for unobserved heterogeneity. In this section we consider semiparametric single-index extensions of such models.

\subsection{Monte Carlo simulations}\label{simulsec}

To evaluate the finite sample performances of our estimator $\widehat{\theta}$ and of the optimal bandwidth $\widehat{h}$, we conduct a simulation experiment with 500 replications.

We consider three explanatory variables $Z=(Z_1,Z_2,Z_3)^\top\sim N(0,\Sigma)$ with $\Sigma=[\sigma_{ij}]_{3\times3}$ and $\sigma_{ij}=0.5^{|i-j|}$.  The regression function is
\begin{equation*}
E(Y\mid Z)=(Z^\top\theta_0)^2+0.5
\end{equation*}
and $\theta _{0}=(\theta _{0}^{(1)},\theta _{0}^{(2)},\theta_{0}^{(3)})^{T}=(1,3,-2)^{T}$. The conditional distribution of $Y$ given $Z$ and $\varepsilon $ is Poisson of mean $r(Z\theta _{0};\theta _{0})\!\cdot \varepsilon $ with $\varepsilon $ independent of $Z$ and distributed according to $\text{Gamma}(0.5,2)$ or $\text{Uniform}(0,2)$. Thus, the conditional variance of $Y$ given $Z$ is given by the function $g\left( r,\alpha \right) =r\left( 1+\alpha r\right) $ with $\alpha _{0}=2$ for $\varepsilon\sim \text{Gamma}(0.5,2)$ and $\alpha=1/3$ for $\varepsilon\sim \text{Uniform}(0,2)$.

For this simulation experiment we generate samples of size $n=200$ and $300$. For the nonparametric part we use a quartic kernel $K\left( u\right) =(15/16)\left(1- u^{2}\right) ^{2}I_{\left[ -1,1\right] }\left( u\right) $ . To estimate the parameter $\theta _{0}$ and the regression $r(\cdot ;\theta_{0})$ we use two semiparametric two-step estimation procedures as defined in section \ref{twostep}: i) A  procedure with a Poisson PML in the first step and a Negative Binomial PML in the second step; let $\widehat{\theta }_{NB-SP}=(1,\widehat{\theta }_{NB-SP}^{(2)},\widehat{\theta }_{NB-SP}^{(3)})^\top$ 
denote the two-step estimator. ii) a procedure with a least-squares method in the first step and a GLS method in the second step; let $\widehat{\theta }_{GLS-SP}=(1,\widehat{\theta }_{GLS-SP}^{(2)}, \widehat{\theta }_{GLS-SP}^{(3)})^\top$  be the two-step estimator. Note that $\widehat{\theta }_{NB-SP}$ and $\widehat{\theta }_{GLS-SP}$ have the same asymptotic variance. In both two-step procedures considered, we estimate $\alpha _{0}$ using the estimator defined in (\ref{nuis}). The bandwidth $h_{n}$ is equal to $3n^{-1/5}.$ We also consider the parametric two-step GLS method as a benchmark. In this case the link function and the variance parameter are considered given; let $\widehat{\theta }_{GLS-P}=(1,\widehat{\theta }_{GLS-P}^{(2)}, \widehat{\theta }_{GLS-P}^{(3)})^\top$ denote the corresponding estimator.

\quad
{\small
\begin{center}
{\footnotesize Table 1. Poisson regression with unobserved heterogeneity $\varepsilon\sim \text{Gamma}(0.5,2)$.  The true conditional variance of $Y$ given $Z$ is $r(Z\theta _{0};\theta _{0})(1+2r(Z\theta _{0};\theta _{0}))$ with $r\left( t;\theta _{0}\right) =t^2+0.5$. The true vector $\theta _{0}$ is $\left( 1,3,-2\right) ^{T}$.  Let  $\widehat{\theta }_{NB-SP}$ and $\widehat{\theta }_{GLS-SP}$ denote the two-step estimators obtained from the Negative Binomial pseudo-likelihood and GLS criterion, respectively. The first step  Poisson PML estimator is denoted by $\widehat{\theta }_{POI-SP}.$ The superscripts  indicate the components of the vectors.}

{\footnotesize
\begin{tabular}{cccccc|cccc}
\hline\hline
& & & &  &  &  &  &  &    \\
$n$ &  & $\widehat{\theta }_{GLS-P}^{(2)}$ & $\widehat{\theta }_{GLS-SP}^{(2)}$&$\widehat{\theta }_{POI-SP}^{(2)}$&$\widehat{\theta }_{NB-SP}^{(2)}$ & $\widehat{\theta }_{GLS-P}^{(3)}$   & $\widehat{\theta }_{GLS-SP}^{(3)}$ & $\widehat{\theta }_{POI-SP}^{(3)}$ & $\widehat{\theta }_{NB-SP}^{(3)}$     \\
 \hline
&  & &  &  &  &  &  &  &    \\
200 & mean &  2.8977 & 2.8019 &  3.0177  & 3.1249 &  -1.9501& -1.6954 &-1.7955   & -2.0520 \\
& std.    &  0.8097 & 0.8986 &   0.9937 &  0.9481 &  0.6268  & 0.5170 &0.6435   & 0.5580 \\
& MSE     &  0.3822 & 0.8467 &  0.9879 &  0.9145 &  0.3929  & 0.3600 & 0.4559  & 0.3167 \\
 \hline
&  &  &  &  & &  &  &  &   \\
300 & mean &  2.9422 &   2.8261& 2.9982  & 3.0758 &  -1.9594 &  -1.7215 &-1.8028  & -1.9569 \\
& std.    &  0.4600 &  0.7741 &  0.9288 &  0.8297 &  0.5002  & 0.4568 &0.5705   & 0.4670 \\
& MSE     &  0.2150 &  0.6295 & 0.8628 & 0.6941 &  0.2519  & 0.2862 & 0.3643  & 0.2199 \\
\hline\hline
\end{tabular}
}
\end{center}
}
\quad

{\small
\begin{center}
{\footnotesize Table 2. The same setup as in Table 1 but with $\varepsilon\sim \text{Uniform}(0,2)$ and the true conditional variance of $Y$ given $Z$ equal to  $ r(Z\theta _{0};\theta _{0})(1+(1/3)r(Z\theta _{0};\theta _{0})).$ }

{\footnotesize
\begin{tabular}{cccccc|cccc}
\hline\hline
&  &  &  &  &  &  &  &  &  \\
$n$ &  &$\widehat{\theta }_{GLS-P}^{(2)}$ & $\widehat{\theta }_{GLS-SP}^{(2)}$ & $\widehat{\theta }_{POI-SP}^{(2)}$ & $\widehat{\theta }_{NB-SP}^{(2)}$
&$\widehat{\theta }_{GLS-P}^{(3)}$ & $\widehat{\theta }_{GLS-SP}^{(3)}$ & $\widehat{\theta }_{POI-SP}^{(3)}$ & $\widehat{\theta }_{NB-SP}^{(3)}$     \\
 \hline
&  &  &  &  &  &  &  &  &  \\
200 & mean & 2.9842 & 2.8460 & 2.9755 &  3.0613  &   -1.9961 & -1.8702 & -1.9094 & -2.0127 \\
& std.    & 0.2505 &   0.4537 &  0.6619 & 0.4917 &0.2551  & 0.2874& 0.4117 & 0.2921 \\
& MSE     & 0.0630 &   0.2295 &  0.4387 &  0.2456 &  0.0651  & 0.0994 & 0.1777  & 0.0855 \\
 \hline
&  &  &  &  &  &  &  &  &  \\
300 & mean & 2.9919&   2.8956 &  2.9422  & 3.0618 & -1.9946 & -1.8999 & -1.8953   & -2.0052 \\
& std.    & 0.2213 & 0.4279 &  0.5753 & 0.3658 & 0.2443  & 0.2647 & 0.3639 &
0.2237 \\
& MSE     & 0.0497 &  0.1940 &  0.3343 &  0.1376 & 0.0597  & 0.0800 & 0.1433  &
 0.0500 \\
\hline\hline

\end{tabular}
}
\end{center}
}

The results on the estimates of the components of $\theta _{0}$ are provided in Table 1 and Table 2. We report the mean, the standard deviation and the estimated mean squared error (MSE) for each component. The two semiparametric estimators that incorporate the information on the conditional variance clearly outperform the semiparametric single-index estimator that ignores that information. Moreover, they behave reasonably well compared to the parametric benchmark.

\subsection{A real data example}

\setcounter{equation}{0}
In order to further illustrate our methodology, we consider a real dataset on recreational trips as presented by Cameron and Trivedi (2013). This data initially collected by Sellar, Stoll and Chavas (1985) is built from a survey that includes the number of recreational boating trips to Lake
Sommerville, Texas. We reproduce below the tables that describe the observed frequencies and the explanatory variables. We do not use all the explanatory variables for estimation since the variables $C1$, $C3$ and $C4$ are almost perfectly correlated in the sample. (Indeed, $Corr(C1,C3)=0.977$, $Corr(C1,C4)=0.987$ and $Corr(C3,C4)=0.964.$) To avoid collinearity problems,
we drop $C3$ and $C4$. We standardize the variables $INC$ and $C1$.

\quad

\begin{center}
{\footnotesize Table 3. The recreational trips data set: actual frequency distribution.}


{\footnotesize
\begin{tabular}{llllllllllll}
\hline\hline
&  &  &  &  &  &  &  &  &  &  &  \\
Number of Trips & 0 & 1 & 2 & 3 & 4 & 5 & 6 & 7 & 8 & 9 & 10 \\
Frequency & 417 & 68 & 38 & 34 & 17 & 13 & 11 & 2 & 8 & 1 & 13 \\ \hline
&  &  &  &  &  &  &  &  &  &  &  \\
Number of Trips & 11 & 12 & 15 & 16 & 20 & 25 & 26 & 30 & 40 & 50 & 88 \\
Frequency & 2 & 5 & 14 & 1 & 3 & 3 & 1 & 3 & 3 & 1 & 1 \\ \hline\hline
\end{tabular}
}
\end{center}

\qquad

\begin{center}
{\footnotesize Table 4. Explanatory variables for the recreational trips counts.}


{\footnotesize
\begin{tabular}{llll}
\hline\hline
&  &  &  \\
Variable & Definition & Mean & Std \\ \hline
&  &  &  \\
$TRIPS$ & Number of recreational boating trips in 1980 & 2.244 & 6.292 \\
& by a sample group &  &  \\
$SO$ & Facility's subjective quality ranking on a scale of 1 to 5 & 1.419 &
1.812 \\
$SKI$ & Equal 1 if engaged in water-skiing at the lake & 0.367 & 0.482 \\
$INC$ & Household income of the head of the group (\$10,000/year) & 0.385 &
0.185 \\
$FC3$ & Equal 1 if user's fee paid at Lake Sommerville & 0.019 & 0.139 \\
$C1$ & Hundreds of dollar expenditure when visiting Lake Conroe & 0.554 &
0.467 \\
$C3$ & Hundreds of dollar expenditure when visiting Lake Somerville & 0.599
& 0.488 \\
$C4$ & Hundreds of dollar expenditure when visiting Lake Houston & 0.560 &
0.461 \\ \hline\hline
\end{tabular}
}
\end{center}

\quad

The model we consider is the one given by equations (\ref{sim})-(\ref{var1}) with $g(r,\alpha)=r(1+\alpha r)$. First, we assume that the regression function is log-linear, that is we consider the standard Negative Binomial Parametric model (NB-P). Next, we no longer assume that the regression function is known and we apply our semiparametric methodology, the semiparametric Negative Binomial pseudo-likelihood procedure.
In the semi-parametric procedures the coefficient of the variable $SO $ is set to 1.  For the nonparametric part we use the quartic kernel $K(u)=(15/16)(1-u^2)^2 I_{[-1,1]}(u)$. The parameter estimates and estimated standard errors are gathered in Table 5, the plot of the estimated  link function  is provided in Figure 1.

\begin{center}
{\footnotesize Table 5. Estimation results: parametric (NB-P) versus semiparametric model ( NB-SP).}


{\footnotesize
\begin{tabular}{lll}
\hline\hline
&  &   \\
Parameters & NB-P &  NB-SP \\
\hline
&  &    \\
$Intercept$ & -1.7452 (0.1441) & .  \\
$SO$ & 0.9017 (0.0430) & 1  \\
$SKI$ & 0.4420 (0.1707) & -0.2489 (0.0405) \\
$INC$ & -0.2245 (0.0906) & 0.1963 (0.0690)\\
$FC3$ & 1.5813 (0.4404) & -0.1399 (0.0702)\\
$C1$ & -0.3258 (0.1018) & -0.2987 (0.0995) \\
$\alpha $ & 2.2983 (0.2210) & 5.5764 \\
$h$ & . & 5.6530  \\
\hline\hline
\end{tabular}
}
\end{center}

\quad
\begin{figure}
\vspace{-1.5cm}
\center
   \includegraphics[width=15cm]{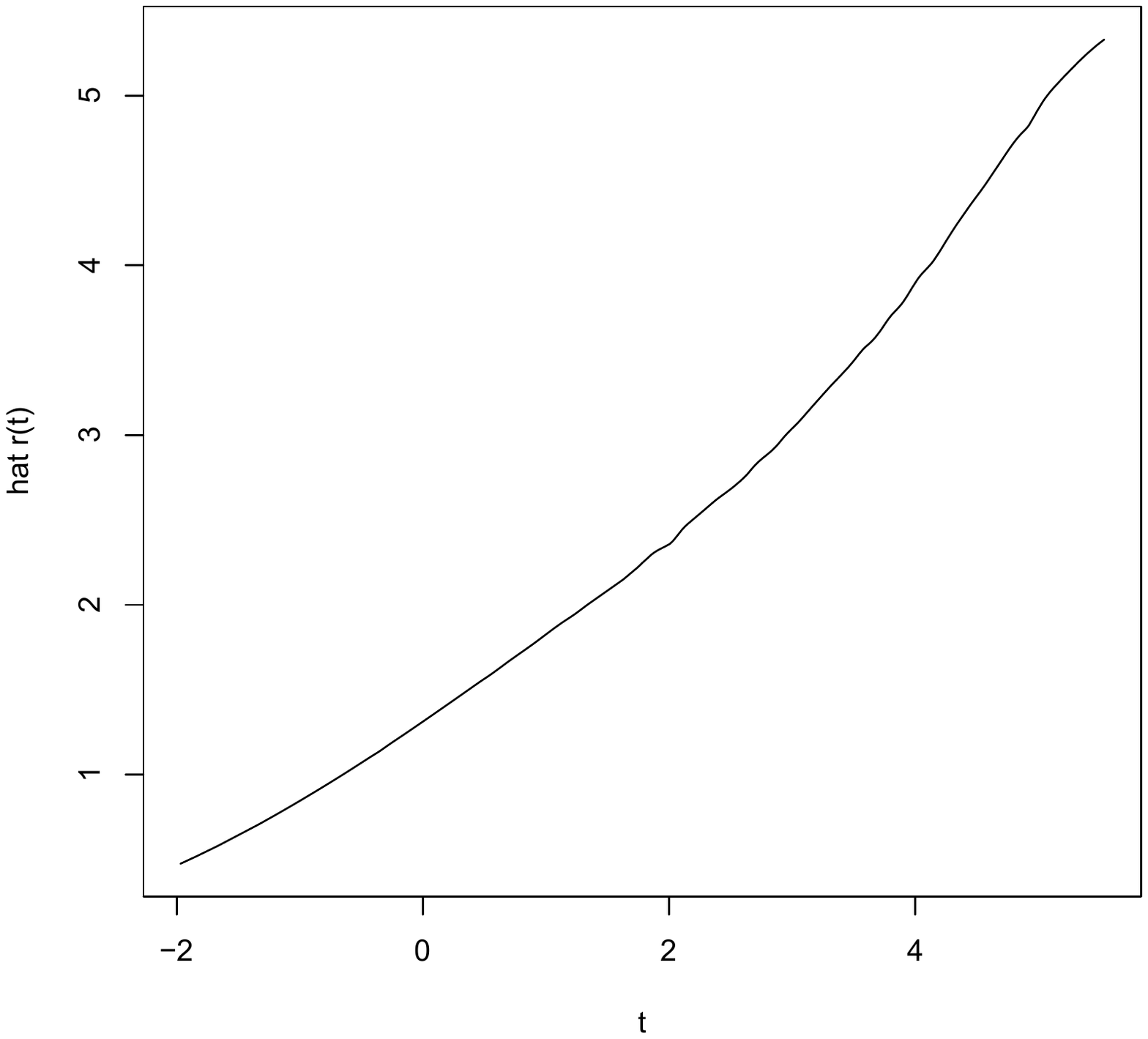}
   \vspace{-1.5cm}
  \caption{The link function}
\end{figure}

Note that the estimate of the coefficient of $SO$ in the parametric model is close to one, while in the semiparametric approach we fixed it to one. Thus the estimated values of the remaining parameters in the parametric and semiparametric cases are almost directly comparable. The results obtained with the semiparametric approach seem more realistic. For instance, the coefficient of $INC$ covariate is positive with NB-SP and the link function is strictly monotone. This suggests that  a higher income more likely induces a larger number of recreational trips.  The NB-P model leads to the opposite conclusion. The reported parametric and semiparametric standard errors cannot
be directly compared on the same basis since we can only compute the standard error of a ratio of parameters in the semiparametric cases. The large bandwidth could be explained by the large conditional variance of the response and a link function with a second derivative close to zero. This leads to  a large constant $(C_2/4C_1)^{1/5}$ in the expression of $h_n^{opt}$, see equation (\ref{a1a2}) above.

In order to evaluate the overall performance of the parametric and semiparametric models and of the estimation methods, we consider various goodness-of-fit measures such as the Pearson statistic, the deviance statistic and the deviance pseudo R-squared statistic. The Pearson
statistics is given by
\begin{equation*}
P=\sum\limits_{i=1}^{n}\frac{\left( Y_{i}-\widehat{r}_{i}\right) ^{2}}{\widehat{\omega }_{i}},
\end{equation*}
where $\widehat{r}_{i}$ is the estimated conditional mean for individual $i$ and $\widehat{ \omega }_{i}$ is the estimated conditional variance computed according to equation (\ref{var1}). The deviance statistic is given by
\begin{equation*}
D=2\sum\limits_{i=1}^{n}\left[ Y_{i}\ln \left( \frac{Y_{i}}{\widehat{r}_{i}} \right) -\left( Y_{i}+1/\widehat{\alpha }\right) \ln \left( \frac{Y_{i}+1/ \widehat{\alpha }} {\widehat{r}_{i}+1/\widehat{\alpha }}\right) \right],
\end{equation*}
with $\widehat{\alpha }$ the estimated value of the nuisance parameter with the values given in the Table 5. Finally, if $\overline{Y}$ denotes the sample mean of the variable $Y$, the
deviance pseudo R-squared statistic is
\begin{equation*}
R_{DEV}^{2}=1-\frac{\sum\limits_{i=1}^{n}\left[ Y_{i}\ln \left( Y_{i}/ \widehat{r}_{i} \right) -\left( Y_{i}+1/\widehat{\alpha }\right) \ln \left( \frac{Y_{i}+1/\widehat{\alpha }}{\widehat{r}_{i}+1/\widehat{\alpha }}\right) \right] }{\sum\limits_{i=1}^{n}\left[ Y_{i}\ln \left(Y_{i}/ \, \overline{Y} \, \right) -\left( Y_{i}+1/\widehat{\alpha }\right) \ln \left( \frac{Y_{i}+1/ \widehat{\alpha }}{\overline{Y}+1/\widehat{\alpha }}\right) \right] }.
\end{equation*}

Another model diagnostic is obtained when comparing fitted probabilities and actual probabilities by the mean of a chi-square type statistic. The statistic we consider is
\begin{equation*}
\xi =n\sum\limits_{j=1}^{J}\frac{\left( \overline{p}_{j}-\widehat{p} _{j}\right) ^{2}} {\overline{p}_{j}},
\end{equation*}
where the possible values of $Y$ are aggregated in $J$ non overlapping cells. \footnote{The chi-square statistic we consider is not necessarily chi-square distributed under the null hypotheses of a well specified model. This is because it does not correctly take into account the estimation error in $\widehat{p}_{j}$. See Andrews (1988) for the general definition of the chi-square goodness-of-fit test statistic in nondynamic regression models. Here, we only use $\xi $ as a crude diagnostic for the three types of fitted probabilities $\widehat{p}_{j}$.} The actual frequency for cell $j$ is denoted $\overline{p}_{j}$ while $\widehat{p}_{j}$ is the corresponding
predicted probability by the model under study. For both methods GLS-SP and NB-SP we used the probabilities of a negative binomial distribution to compute $\widehat{p}_{j}.$ We consider seven cells corresponding to the values $TRIP=0,...,5$ and $TRIP>5$. All the results are summarized in Table 6. The semiparametric model performs better than the parametric model.
We also give the estimators of the probability in Table 7. We can see that our estimators are close to the empirical probability of $TRIP$. The semiparametric approach greatly improves the standard  parametric modeling.

\quad

\begin{center}
{\footnotesize Table 6. Goodness-of-fit statistics: $P-$Pearson statistic, $D-$deviance statistic, $R_{DEV}^{2}-$deviance pseudo R-squared statistic and $\xi -$chi-square statistic.}


{\footnotesize
\begin{tabular}{lll}
\hline\hline
&  &    \\
& NB-P & NB-SP \\ \hline
&  &   \\
$P$ &5296.506 & 608.7212  \\
$D$ & 1158.41 & 405.1771  \\
$R_{DEV}^{2}$ & 0.4780 & 0.1886 \\
$\xi $ & 968.9416 & 3.1922 \\
 \hline\hline
\end{tabular}
}
\qquad
\end{center}
\quad
\begin{center}
{\footnotesize Table 7. Empirical probability and estimate probability}


{\footnotesize
\begin{tabular}{llllllll}
\hline
\hline
$TRIPS$& 0&1&2&3&4&5&$>5$ \\
\hline
Empirical probability&0.6327&0.1031&0.0576&0.0515&0.0257&0.0197&0.1092\\
NB--P& 0.1111& 0.1572& 0.1596& 0.1407& 0.1147& 0.0889&0.2273\\
NB--SP&0.6314& 0.1045& 0.0568& 0.0381& 0.02797& 0.0215&0.1194\\
\hline
\hline
\end{tabular}
}
\quad
\end{center}

\section{Conclusion}

\label{concl}

\setcounter{equation}{0}

We consider a semiparametric single-index model (SIM) where an additional second order moment condition is specified. To estimate the parameter of interest $\theta $ we introduce a two-step semiparametric pseudo-maximum likelihood (PML) estimation procedure based on linear exponential families with nuisance parameter densities. This procedure extends the quasi-generalized pseudo-maximum likelihood method proposed by Gouri\'{e}roux, Monfort and Trognon (1984a, 1984b). We also provide a natural rule for choosing the bandwidth of the nonparametric smoother appearing in the estimation procedure. The idea is to maximize the pseudo-likelihood of the second step simultaneously in $\theta $ and the smoothing parameter $h$. The rate of the bandwidth is allowed to lie in a range between $n^{-1/4}$ and $n^{-1/8}$. We derive the asymptotic behavior of $\widehat{\theta }$, the two-step semiparametric PML we propose. If the SIM condition holds, then $\widehat{\theta }$ is $\sqrt{n}-$asymptotically normal. We also provide a consistent estimator of its variance. When the SIM condition holds and the conditional variance is correctly specified, then $\widehat{\theta }$ has the best variance amongst the semiparametric PML estimators. The `optimal' bandwidth $\widehat{h}$ obtained by joint maximization of the pseudo-likelihood function in the second step is shown to be equivalent to the minimizer of a weighted cross-validation function. From this we deduce that $n^{1/5} \widehat{h}$ converges to a positive constant, in probability. In particular, our optimal bandwidth $\widehat{h}$ has the rate expected when
estimating a twice differentiable regression function nonparametrically. We conduct a simulation experiment in which the data were generated using a Poisson single-index regression model with multiplicative unobserved heterogeneity. The simulation confirms the significant advantage of estimators that incorporate the information on the conditional variance. We also applied our semiparametric approach to a benchmark real count data set and we obtain a much better fit than the standard parametric regression models for count data.

\newpage

\appendix

\section{Appendix: Assumptions}

\label{assu}\setcounter{equation}{0}

\setcounter{appen}{0}

Let $\Theta =\left\{ 1\right\} \times \widetilde{\Theta }$ with $\widetilde{%
\Theta }$ a compact subset of $\mathbb{R}^{d-1}$ with nonvoid interior.
Depending on the context, $\Theta $ is considered a subset of $\mathbb{R}%
^{d-1}$ or a subset of $\mathbb{R}^{d}.$

\begin{assumption}
\label{asiid}The observations $\left( Y_{1},Z_{1}^{T}\right) ^{T},\ldots
,\left( Y_{n},Z_{n}^{T}\right) ^{T}$ are independent copies of a random
vector $\left( Y,Z^{T}\right) ^{T}\in \mathbb{R}^{d+1}.$
\end{assumption}

\begin{assumption}
Let $r\left( t;\theta \right) =E\left(  Y\mid Z^{T}\theta =t\right)  .$ There
exists a unique $\theta _{0}$ interior point of $\Theta $ such that $E\left(
Y\mid Z\right) =E\left( Y\mid Z^{T}\theta _{0}\right) =r\left( Z^{T}\theta
_{0};\theta _{0}\right) .$
\end{assumption}

\begin{assumption}
\label{as41}For every $\theta \in \Theta $, the random variable $Z^{T}\theta
$ admits a density $f(\cdot ;\theta )$ with respect to the Lebesgue measure
on $\mathbb{R}$.
\end{assumption}

\begin{assumption}
\label{as42}$E\left[ \exp \left( \lambda \left\| Z\right\| \right) \right]
<\infty ,$ for some $\lambda >0$. Moreover, $E(Y^{4+\varepsilon })<\infty ,$ for some $\varepsilon >0$.
\end{assumption}

\begin{assumption}
\label{colin} With probability one, the matrix $(1,Z^{T})^{T}(1,Z^{T})$ is
positive definite.
\end{assumption}

\begin{assumption}
\label{cc} There exists $c_{0}>0$ and a positive integer $k_{0}$ such that,
for any $\theta \in \Theta $ and $0<c\leq c_{0}$, the set $\left\{
t:f(t;\theta)=c\right\} $ has at most $k_{0}$ elements.
\end{assumption}

The last two assumptions ensure that $P\left( f(Z^{T}\theta _{0};\theta
_{0})=c\right) =0,$ for any $0<c\leq c_{0}.$

\quad

\textbf{CONDITION L} A function $g:\Theta \times \mathbb{R}\rightarrow
\mathbb{R}$ is said to satisfy \emph{Condition L} if, for any $\Lambda $ a
compact set on the real line, there exists $B>0$ and $b\in (0,1]$ such that
\begin{equation*}
\left| g\left( \theta ,t\right) -g\left( \theta ^{\prime },t^{\prime
}\right) \right| \leq B\left\| \left( \theta ,t\right) -\left( \theta
^{\prime },t^{\prime }\right) \right\| ^{b},\quad \quad \theta ,\theta
^{\prime }\in \Theta ,\quad t,t^{\prime }\in \Lambda .
\end{equation*}

\begin{assumption}
\label{as45} a) The function $\left( \theta ,t\right) \rightarrow f\left(
t;\theta \right) \geq 0,$ $\theta \in \Theta ,$ $t\in \mathbb{R},$ satisfies
a Lipschitz condition, that is there exists $a\in (0,1]$ and $C>0$ such that
\begin{equation*}
\left| f\left( t;\theta \right) -f\left( t^{\prime };\theta ^{\prime
}\right) \right| \leq C\,\left\| \left( \theta ,t\right) -\left( \theta
^{\prime },t^{\prime }\right) \right\| ^{a}\qquad for\quad \theta ,\theta
^{\prime }\in \Theta \qquad and\quad t,t^{\prime }\in \mathbb{R}.
\end{equation*}

b) The function $\left( \theta ,t\right) \rightarrow r\left( t;\theta
\right) ,$ $\theta \in \Theta ,$ $t\in \mathbb{R},$ satisfies \emph{%
Condition L}$.$

c) For any $\theta \in \Theta ,$ the functions $t\rightarrow \gamma \left(
t;\theta \right) $ and $t\rightarrow f\left( t;\theta \right) $ are twice
differentiable.\ Let $\gamma ^{\prime \prime }\left( t;\theta \right) $ and $%
f^{\prime \prime }\left( t;\theta \right) $ denote the second order
derivatives. The functions $\left( \theta ,t\right) \rightarrow \gamma
^{\prime \prime }\left( t;\theta \right) $ and $\left( \theta ,t\right)
\rightarrow f^{\prime \prime }\left( t;\theta \right) ,$ $\theta \in \Theta
, $ $t\in \mathbb{R},$ satisfy \emph{Condition L} with $b=1$\emph{.}

d) For any $\theta \in \Theta $ and any component $Z^{(j)}$ of $Z,$ the
functions $t\rightarrow E\left( Z^{(j)}\,|\,Z^{T}\theta =t\right) $ and $%
t\rightarrow E\left( Y\,Z^{(j)}\,|\,Z^{T}\theta =t\right) $ are twice
differentiable and their second order derivatives satisfy \emph{Condition L}
with $b=1$\emph{.}

e) For any $t\in \mathbb{R},$ the function $\theta \rightarrow r\left(
t;\theta \right) $ is twice continuously differentiable and, for any $\theta
\in \Theta ,$ the functions $t\rightarrow \partial _{\theta }r\left(
t;\theta \right) $ and $t\rightarrow \partial _{\theta \theta }^{2}r\left(
t;\theta \right) $ are continuous. Moreover, the function $\left( \theta
,t\right) \rightarrow \partial _{\theta }r\left( t;\theta \right) $ satisfy
\emph{Condition L} with $b=1$\emph{.}
\end{assumption}

Let $v\left( t;\theta \right) =Var\left( Y\mid Z^{T}\theta =t\right) $ be
the conditional variance of $Y$ given $Z^{T}\theta =t.$

\begin{assumption}
\label{as46}The function $(\theta ,t)\rightarrow v\left( t;\theta \right) $
satisfies \emph{\ Condition L}.
\end{assumption}

Consider the functions $B,C:R\times N\rightarrow \mathbb{R}$, with $\mathcal{%
Y}$, $R,$ $N\subset \mathbb{R}$. Define $\Lambda =\bigcup_{\theta \in \Theta
}\{t:f(t;\theta )\geq c\}$, with $c$, $\delta >0$, and
\begin{equation*}
D(c,\delta )=\{r:\exists \,(\theta ,\,t)\in \Theta \times \Lambda \,\,\,%
\text{such that}\,\,|r-r(t;\theta )|\leq \delta \}.
\end{equation*}

\begin{assumption}
\label{asssup}If $c>0$, there exists $\delta >0$ such that $D(c,\delta )$ is
strictly included in $R$.
\end{assumption}

\begin{assumption}
\label{as410}The kernel function $K\left( \cdot \right) $ is differentiable,
symmetric, positive and compactly supported. Moreover, $K\left( \cdot
\right) $ and the derivative $K^{\prime }\left( \cdot \right) $ are of
bounded variation.
\end{assumption}

\quad

Up to a term depending only on $y$ and $\alpha ,$ the three arguments
function $\psi \left( \cdot ,\cdot ;\cdot \right) $ involved in equation (%
\ref{deeff}) is defined as
\begin{equation*}
\psi \left( y,r;\alpha \right) =B(r,\alpha )+C(r,\alpha )y
\end{equation*}%
where $l(y\mid r,\alpha )=\exp \left[ B(r,\alpha )+C(r,\alpha )y+D(y,\alpha )%
\right] $ is a LEFN density with mean $r$ and variance $\left[ \partial
_{r}C(r,\alpha )\right] ^{-1}.$

\quad

\begin{assumption}
\label{as49} The functions $B\left( r,\alpha \right) $ and $C\left( r,\alpha
\right) $ are twice differentiable in the first argument. Moreover, for any $%
c$ and $\delta >0$ for which $D(c,\delta )$ is strictly included in $R$,
there exists a constant $M$ such that
\begin{equation*}
\sup_{r\in D(c,\delta ),\,\alpha \in N}\left( |\partial _{rr}^{2}G(r,\alpha
)|+|\partial _{r}G(r,\alpha )|\right) \leq M,
\end{equation*}%
\begin{equation*}
\sup_{r,r^{\prime }\in D(c,\delta ),\,\alpha ,\alpha ^{\prime }\in N}\left|
\partial _{rr}^{2}G(r,\alpha )-\partial _{rr}^{2}G(r^{\prime },\alpha
^{\prime })\right| \leq M\left( |r-r^{\prime }|+\left| \alpha -\alpha
^{\prime }\right| \right) ,
\end{equation*}%
where $G$ stands for $B$ or $C$. The functions $\partial _{r}B(r;\alpha ) $
and $\partial _{r}C(r;\alpha )$ are continuously differentiable in $\alpha .$
\end{assumption}

\begin{assumption}
\label{as51}For any $c$ and $\delta >0$ for which $D(c,\delta )$ is strictly
included in $R,$ we have $\partial _{r}C(r,\alpha )>0,$ $\forall r\in
D(c,\delta ),$ $\forall \alpha \in N.$
\end{assumption}

\quad

Assumption \ref{as51} ensures that the $(d-1)\times (d-1)$ matrix%
\begin{eqnarray*}
J &=&-E\left[ \partial _{\theta \theta ^{T}}^{2}\psi \left( Y,r\left(
Z^{T}\theta _{0};\theta _{0}\right) ;\;\alpha ^{\ast }\right) \ I_{A}\left(
Z\right) \right] \\
&=&E\left[ \partial _{r}C\left( r\left( Z^{T}\theta _{0};\theta _{0}\right)
;\alpha ^{\ast }\right) \partial _{\theta }r\left( Z^{T}\theta _{0};\theta
_{0}\right) \partial _{\theta }r\left( Z^{T}\theta _{0};\theta _{0}\right)
^{T}I_{A}\left( Z\right) \right]
\end{eqnarray*}%
is positive definite.

Let us notice that the asymptotic results remain valid even if the function $%
\psi \left( y,r;\alpha \right) $ is not the logarithm of a LEFN.\ It
suffices to adapt Assumption \ref{as49}, to suppose that there exists $%
F\left( \cdot ;\cdot \right) $ such that $\psi \left( y,r;\alpha \right)
\leq F\left( y;\alpha \right) ,$ $\forall r\in R,$ to ensure that $J$ is
positive definite and to assume that, for any $\alpha$, $E\left[ \partial
_{2}\psi \left( Y,\;r\left( Z^{T}\theta _{0};\theta _{0}\right) ;\;\alpha
\right) \mid Z\right] =0$ and
\begin{equation*}
E\left[ \partial _{\theta }\partial _{2}\psi \left( Y,\;r\left( Z^{T}\theta
_{0};\theta _{0}\right) ;\;\alpha \right) \mid Z^{T}\theta _{0}\right] =0 .
\end{equation*}

\section{Appendix: Technical lemmas}

\label{appconv}

\setcounter{appen}{0} \setcounter{equation}{0}

Let $H_{n}=\left[ n^{-(1/2-\varepsilon )},\,n^{-\varepsilon
}\right] ,$ with $0<\varepsilon <1/2,$ and $\Theta _{n}=\left\{
\theta :\left\| \theta -\theta _{0}\right\| \leq d_{n}\right\} $,
with $ d_{n}\rightarrow 0 $. We use $C$ to denote a positive
constant, not necessarily the same at each occurrence.

\begin{lemma}
\label{lemrate} Assume that the kernel $K$ is a symmetric, positive,
compactly supported function of bounded variation. Suppose that the map $%
\left( \theta ,t\right) \rightarrow f\left( t;\theta \right) \geq 0,$ $%
\theta \in \Theta ,$ $t\in \mathbb{R},$ satisfies a Lipschitz condition,
that is there exists $a\in (0,1]$ and $C>0$ such that
\begin{equation}
\left| f\left( t_{1};\theta _{1}\right) -f\left( t_{2};\theta _{2}\right)
\right| \leq C\,\left\| \left( \theta _{1},t_{1}\right) -\left( \theta
_{2},t_{2}\right) \right\| ^{a}\qquad for\quad \theta _{1},\theta _{2}\in
\Theta \qquad and\quad t_{1},t_{2}\in \mathbb{R}.  \label{lip1}
\end{equation}
Then
\begin{equation*}
\max_{1\leq i\leq n}\,\sup_{\theta ,\,z,\,h\in H_{n}}\left| \widehat{f}%
_{h}^{i}\left( z^{T}\theta ;\theta \right) -f\left( z^{T}\theta ;\theta
\right) \right| =O_{P}\left( h^{-1}n^{-1/2}\right) +O\left( h^{a}\right) .
\end{equation*}
\end{lemma}

\quad

The proof of Lemma \ref{lemrate} can be distilled from many existing results
(e.g., Andrews (1995), Sherman (1994b), Delecroix, Hristache and Patilea
(2006)) and therefore it will be omitted.

\begin{lemma}
\label{idic} a) If $\delta >0,$ then
\begin{equation*}
\sup_{\theta \in \Theta _{n},\,h\in H_{n}}\left| I_{\left\{ z:\,\widehat{f}%
_{h}^{i}\left( z^{T}\theta ;\theta \right) \geq c\right\}
}(Z_{i})-I_{A}(Z_{i})\right| \leq I_{A^{\delta }}(Z_{i})+I_{(\delta ,\infty
)}(G_{n}),\qquad 1\leq i\leq n,
\end{equation*}
where $A^{\delta }=\left\{ z:\left| f\left( z^{T}\theta _{0};\theta
_{0}\right) -c\right| \leq \delta \right\} $ and
\begin{equation*}
G_{n}=\max_{1\leq i\leq n}\,\sup_{\theta \in \Theta _{n},\,h\in
H_{n}}\,\left| \widehat{f}_{h}^{i}\left( Z_{i}^{T}\theta ;\theta \right)
-f\left( Z_{i}^{T}\theta _{0};\theta _{0}\right) \right| .
\end{equation*}

b) Suppose that $K\left( \cdot \right) $ and $f\left( \cdot ;\cdot \right) $
satisfy the assumptions of Lemma \ref{lemrate} for some $a,C>0$. Moreover,
assume that either i) $Z$ is bounded and $d_{n}\rightarrow 0$ or, ii) $E%
\left[ \exp \left( \lambda \left\| Z\right\| \right) \right] <\infty $ for
some $\lambda >0$ and $d_{n}=o\left( 1/\ln n\right) $ (with $d_{n}$ from the
definition of $\Theta _{n}$)$.$ Let $\delta _{n}\rightarrow 0$ such that $%
\delta _{n}/n^{-a\varepsilon }\rightarrow \infty $ and either i) $\delta
_{n}d_{n}^{-a}\rightarrow \infty $ if $Z$ is bounded or, ii) $\delta _{n}%
\left[ d_{n}\ln n\right] ^{-a}\rightarrow \infty $. Then $I_{(\delta
_{n},\infty )}(G_{n})=o_{P}\left( n^{-\alpha }\right) ,$ $\forall \alpha >0.$
\end{lemma}

\begin{proof}
a) We have
\begin{equation*}
\left| I_{\left\{ z:\,\widehat{f}_{h}^{i}\left( z^{T}\theta ;\theta \right)
\geq c\right\} }(Z_{i})-I_{A}(Z_{i})\right| \leq I_{\left\{ z:\,\widehat{f}%
_{h}^{i}\left( z^{T}\theta ;\theta \right) \geq c\right\} \setminus
A}(Z_{i})+I_{A\setminus \left\{ z:\,\widehat{f}_{h}^{i}\left( z^{T}\theta
;\theta \right) \geq c\right\} }(Z_{i}).
\end{equation*}
For any $\theta ,h$ and $\delta ,$ we can write
\begin{equation*}
\left\{ \widehat{f}_{h}^{i}\left( Z_{i}^{T}\theta ;\theta \right) \!\geq
c\right\} \!\setminus \!A\!\subset \!\left\{ \widehat{f}_{h}^{i}\left(
Z_{i}^{T}\theta ;\theta \right) \!\geq c,\,f\left( Z_{i}^{T}\theta
_{0};\theta _{0}\right) \!<c\!-\!\delta \right\} \!\cup \!\left\{
c\!-\!\delta \leq \!f\left( Z_{i}^{T}\theta _{0};\theta _{0}\right)
\!<c\right\}
\end{equation*}
and
\begin{equation*}
A\!\setminus \!\left\{ \widehat{f}_{h}^{i}\left( Z_{i}^{T}\theta ;\theta
\right) \geq c\right\} \!\subset \!\left\{ \widehat{f}_{h}^{i}\left(
Z_{i}^{T}\theta ;\theta \right) \!<c,\,f\left( Z_{i}^{T}\theta _{0};\theta
_{0}\right) \!\geq c\!+\!\delta \right\} \!\cup \!\left\{ c\leq \!f\left(
Z_{i}^{T}\theta _{0};\theta _{0}\right) \!<c\!+\!\delta \right\}
\end{equation*}
which proves the inequality.\

b) It suffices to prove that $P\left( G_{n}>\delta _{n}\right) \rightarrow
0. $ First consider the case of unbounded $Z.$ Note that, for any $z$ and $%
\theta ,$
\begin{equation*}
\left| f\!\left( z^{T}\theta _{1};\theta _{1}\right) \!-f\!\left(
z^{T}\theta _{2};\theta _{2}\right) \right| \!\leq \!C\left( \left|
z^{T}\theta _{1}-z^{T}\theta _{2}\right| ^{2}+\!\left\| \theta _{1}-\theta
_{2}\right\| ^{2}\right) ^{a/2}\!\leq \!C\left( 1\!+\!\left\| z\right\|
\right) ^{a}\left\| \theta _{1}-\theta _{2}\right\| ^{a}\!.
\end{equation*}%
Combine this inequality and Lemma \ref{lemrate} and write
\begin{eqnarray*}
G_{n}\! &\leq &\max_{1\leq i\leq n}\,\sup_{\theta \in \Theta ,\,h\in
H_{n}}\,\left| \widehat{f}_{h}^{i}\left( Z_{i}^{T}\theta ;\theta \right)
-f\left( Z_{i}^{T}\theta ;\theta \right) \right| \\
&&+\max_{1\leq i\leq n}\sup_{\theta \in \Theta _{n}}\,\left| f\left(
Z_{i}^{T}\theta ;\theta \right) -f\left( Z_{i}^{T}\theta _{0};\theta
_{0}\right) \right| \\
&\leq &\max_{1\leq i\leq n}\sup_{\theta \in \Theta ,\,h\in
H_{n},\,z}\,\left| \widehat{f}_{h}^{i}\left( z^{T}\theta ;\theta \right)
-f\left( z^{T}\theta ;\theta \right) \right| +C\left\| \theta -\theta
_{0}\right\| ^{a}\,\max_{1\leq i\leq n}\,\left( 1+\left\| Z_{i}\right\|
\right) ^{a} \\
&=&O\left( n^{-a\varepsilon }\right) +O_{P}\left( n^{-\varepsilon }\right)
+O\left( d_{n}^{a}\right) \max_{1\leq i\leq n}\,\left( 1+\left\|
Z_{i}\right\| \right) ^{a}.
\end{eqnarray*}%
On the other hand, we can write
\begin{eqnarray*}
P\left( d_{n}^{a}\max_{1\leq i\leq n}\,\left( 1+\left\| Z_{i}\right\|
\right) ^{a}>\delta _{n}\right) &\leq &\sum_{i=1}^{n}P\left( \left(
1+\left\| Z_{i}\right\| \right) ^{a}>\delta _{n}/d_{n}^{a}\right) \\
&=&nP\left[ \exp \left( \lambda (1+\left\| Z_{i}\right\| )\right) >\exp
\left( \lambda \delta _{n}^{1/a}/d_{n}\right) \right] \\
&\leq &n\frac{e^{\lambda }E\left[ \exp \left( \lambda \left\| Z_{i}\right\|
\right) \right] }{\exp \left( \lambda \delta _{n}^{1/a}/d_{n}\right) }.
\end{eqnarray*}%
Since $\delta _{n}^{1/a}/\left( d_{n}\ln n\right) $ and $\delta
_{n}/n^{-a\varepsilon }\rightarrow \infty ,$ deduce that $P\left(
G_{n}>\delta _{n}\right) \rightarrow 0.$

If $Z$ lies in a compact, condition (\ref{lip1}) implies that for any $z$ in
the support of $Z$,
\begin{equation*}
\left| f\left( z^{T}\theta _{1};\theta _{1}\right) -f\left( z^{T}\theta
_{2};\theta _{2}\right) \right| \leq C\left\| \theta _{1}-\theta
_{2}\right\| ^{a},\qquad \theta _{1},\theta _{2}\in \Theta ,
\end{equation*}
with $C>0$ some constant independent of $z.$ In this case
\begin{equation*}
G_{n}=O\left( n^{-a\varepsilon }\right) +O_{P}\left( n^{-\varepsilon
}\right) +O\left( d_{n}^{a}\right) .
\end{equation*}
Thus $P\left( G_{n}>\delta _{n}\right) \rightarrow 0$ provided that $\delta
_{n}\rightarrow 0$ such that $\delta _{n}/n^{-a\varepsilon }$ and $\delta
_{n}/d_{n}^{a}\rightarrow \infty $
\end{proof}

\quad

The proofs of the following three lemmas are lengthy and technical.\ These
proofs are provided in Delecroix, Hristache and Patilea (2006) and therefore
it will be omitted herein. The key ingredients for the three proofs are the
results on uniform rates of convergence for $U-$processes indexed by
Euclidean families; see Sherman (1994a). See also Pakes and Pollard (1989)
for the definition and the properties of Euclidean families of functions.

The first of the three lemmas is a refined version of a standard result for
cross-validation in nonparametric regression (e.g., H\"{a}rdle and Marron
(1985)). The result holds uniformly in $\theta $,$\alpha $ and for $h$ in $%
H_{n}=\left[ n^{-(1/2-\varepsilon )},\,n^{-\varepsilon }\right] ,$ with $%
0<\varepsilon <1/2.$

\begin{lemma}
\label{mise} Suppose that Assumptions \ref{asiid} to \ref{as46} hold. Fix
some small $c>0$ and let $\Lambda =\bigcup_{\theta \in \Theta
}\{t:f(t;\theta )\geq c\}$. Consider a family of functions $(y,t)\rightarrow
w_{\theta ,\alpha }(y,t)$, $\theta \in \Theta ,$ $\alpha \in N$\ for which
there exist a real-valued function $B(\cdot )$ with $E[B(Y)^{4+\varepsilon
}]<\infty $, for some $\varepsilon >0$, and $b^{\prime }\in (0,1]$ such
that, for each $y$
\begin{equation*}
|w_{\theta ,\alpha }(y,t)-w_{\theta ^{\prime },\alpha ^{\prime
}}(y,t^{\prime })|\leq B(y)\left\| (\theta ^{T},\alpha ,t)^{T}-(\theta
^{\prime T},\alpha ^{\prime },t^{\prime })^{T}\right\| ^{b^{\prime }}
\end{equation*}
for any $\theta ,\theta ^{\prime }\in \Theta $, $t,t^{\prime }\in \Lambda$,
and $\alpha ,\alpha ^{\prime }\in N$. Moreover, there exist $\theta ,\alpha $
and $\widetilde{B}(\cdot )$ such that $sup_{t\in \Lambda }|w_{\theta ,\alpha
}(\cdot ,t)|\leq \widetilde{B}(\cdot )$ and $E[\widetilde{B}%
(Y)^{4+\varepsilon }]<\infty $.

For $(\theta ,h)\!\in \!\Theta \!\times \!H_{n}$ and $\alpha \in N,$ define
\begin{equation*}
U\!\left( \theta ,h;\alpha \right) =\!\dfrac{1}{n}\sum\limits_{i=1}^{n}w_{%
\theta ,\alpha }\!\left( Y_{i},\!Z_{i}^{T}\theta \right) \!\left[ \hat{r}%
_{h}^{i}\!\left( Z_{i}^{T}\theta ;\theta \right) \!-\!r\!\left(
Z_{i}^{T}\theta ;\theta \right) \right] ^{2}\!I_{\left\{ z:\,f\left(
z^{T}\theta ;\theta \right) \geq c\right\} }\left( Z_{i}\right) \quad ;
\end{equation*}%
the kernel is a continuous probability density function $K$ with the support
in $[-1,1].$ Moreover, $K$ is of bounded variation and symmetric. Then,
\begin{equation*}
U\left( \theta ,h;\alpha \right) =-h^{4}C_{1}\left( \theta ;\alpha \right) -%
\frac{1}{nh}C_{2}\left( \theta ;\alpha \right) +\rho \left( \theta ,h;\alpha
\right) ,
\end{equation*}%
where
\begin{eqnarray}
C_{1}(\theta ;\alpha )\!\!\!\! &=&\!\!\! \dfrac{K_{1}^{2}}{4}\!E\left\{
\!\!-w_{\theta ,\alpha }\! \left( Y,Z^{T}\theta \right) \!\!\left[ r^{\prime
\prime }\!\left( Z^{T}\theta ;\theta \right) \!+\frac{2r^{\prime }\!\left(
Z^{T}\theta ;\theta \right) \!f^{\prime }\! \left( Z^{T}\theta ;\theta
\right) }{f\left( Z^{T}\theta ;\theta \right) }\right] ^{2}\!\!\!I_{\left\{
z:f\left( z^{T}\theta ;\theta \right) \geq c\right\} }\!\left( Z\right) \!
\! \right\} \!,  \notag \\
C_{2}\left( \theta ;\alpha \right) \!\!\!\! &=&\!\!\! K_{2}\ E\left\{ -\,\,%
\frac{w_{\theta ,\alpha }\left( Y,Z^{T}\theta \right) }{f\left( Z^{T}\theta
;\theta \right) }\ v\left( Z^{T}\theta ;\theta \right) \,I_{\left\{
z:\,f\left( z\theta ;\theta \right) \geq c\right\} }\left( Z\right) \right\}
\!,  \notag
\end{eqnarray}%
with $K_{1}=\int u^{2}K\left( u\right) du$, $K_{2}=\int K^{2}\left( u\right)
du$ and
\begin{equation*}
\sup_{\theta \in \Theta ,\,h\in H_{n},\alpha \in N}\rho \left( \theta
,h;\alpha \right) =o_{P}\left( h^{4}+\left( nh\right) ^{-1}\right) .
\end{equation*}
\end{lemma}

\quad

\begin{lemma}
\label{mise2} Assume that the conditions of Lemma \ref{mise} hold. Let
\begin{equation*}
\widetilde{T}(\theta ,h;\alpha )=\dfrac{1}{n}\sum\limits_{i=1}^{n}\ \pi
\left( Y_{i},Z_{i}^{T} ;\alpha \right) \left[ \hat{r}_{h}^{i}\left(
Z_{i}^{T}\theta ;\theta \right) -r\left( Z_{i}^{T}\theta ;\theta \right) %
\right] \ I_{\left\{ z:\,f\left( z^{T}\theta _{0};\theta _{0}\right) \geq
c\right\} }\left( Z_{i}\right) ,
\end{equation*}%
where, for some $\varepsilon >0$, $E\left[ \pi \left( Y,Z;\alpha \right)
\mid Z\right] =0$ and $E\left[ \left| \pi \left( Y,Z;\alpha\right) \right|
^{4+\varepsilon }I_{\left\{ z:\,f\left( z^{T}\theta _{0};\theta _{0}\right)
\geq c\right\} }\left( Z\right) \right] <\infty $, for any $\alpha$. Then,
\begin{equation*}
\widetilde{T}(\theta ,h;\alpha )=o_{P}\left( h^{4}+n^{-1}h^{-1}\right) ,
\end{equation*}%
uniformly in $h\in \mathcal{H}_{n}$, in $\theta \in \Theta _{n}$ and in $%
\alpha \in N.$
\end{lemma}

\quad

The following lemma will provide the order of the reminder term in the
decomposition (\ref{deco}).\ The proof relies on orthogonality conditions
like (\ref{orth1}) and (\ref{orth2}).\

\begin{lemma}
\label{rr} Let $A=\left\{ z:f\left( z^{T}\theta _{0};\theta _{0}\right) \geq
c\right\} \subset \mathbb{R}^{d},$ for some $c>0.$ Let $\phi \left(
y,r;\alpha \right) =B\left( r;\alpha \right) +C\left( r;\alpha \right) y$
such that such that $B^{\prime }\left( r;\alpha \right) +C^{\prime }\left(
r;\alpha \right) r\equiv 0.$ Suppose that the Assumptions \ref{asiid}, \ref%
{as41}, \ref{as42} and \ref{as45} to \ref{as49} hold. Then
\begin{eqnarray*}
R\left( \theta ,h; \alpha \right) &=&\dfrac{1}{n}\sum\limits_{i=1}^{n}\left[
\phi \left( Y_{i},\hat{r}_{h}^{i}\left( Z_{i}^{T}\theta ;\theta \right)
;\alpha \right) -\phi \left( Y_{i},r\left( Z_{i}^{T}\theta ;\theta \right)
;\alpha \right) \right] \ I_{A}\left( Z_{i}\right) \\
&&-\dfrac{1}{n}\sum\limits_{i=1}^{n}\left[ \phi \left( Y_{i},\hat{r}%
_{h}^{i}\left( Z_{i}^{T}\theta _{0};\theta _{0}\right) ;\alpha \right) -\phi
\left( Y_{i},r\left( Z_{i}^{T}\theta _{0};\theta _{0}\right) ;\alpha \right) %
\right] \ I_{A}\left( Z_{i}\right) \\
&& \\
&=&\left[ O_{P}\left( h^{4}\right) +O_{P}\left( \frac{1}{nh^{2}}\right)
+O_{P}\left( \frac{h^{2}}{\sqrt{n}}\right) +O_{P}\left( \frac{1}{n\sqrt{n}%
h^{4}}\right) \right] \times O_{P}\left( \left\| \theta -\theta _{0}\right\|
\right) \\
&&+\left[ O\left( h^{2}\right) +O_{P}\left( \frac{1}{\sqrt{n}h^{2}}\right) %
\right] \times O_{P}\left( \left\| \theta -\theta _{0}\right\| ^{2}\right) \\
&&
\end{eqnarray*}
when $n\rightarrow \infty $, uniformly in $\alpha ,$ uniformly in $h\in %
\left[ n^{-(1/2-\varepsilon )},\,n^{-\varepsilon }\right] ,$ with $%
0<\varepsilon <1/2,$ and uniformly in $\theta \in \Theta _{n}.$
\end{lemma}

\section{Appendix: proofs of the main results}

\label{proof}

\begin{proof}[Proof of Theorem \ref{param1}]
Assume for the moment that $\left( \widehat{\theta },\widehat{h}\right) $
are defined by maximization of $\widehat{S}\left( \theta ,h;\widetilde{%
\alpha }_{n}\right) $ which is defined with the fixed trimming $I_{A}(\cdot
) $ (see equation (\ref{infeasi})). At the end of the proof we show that the
same conclusions hold for $\left( \widehat{\theta },\widehat{h}\right) $
defined in (\ref{deeff}) with the data-driven trimming.

\emph{Part I : }$\sqrt{n}-$\emph{asymptotic normality of }$\widehat{\theta }%
. $\emph{\ }By the decomposition (\ref{deco}) we have
\begin{equation*}
\widehat{S}\left( \theta ,h;\widetilde{\alpha }_{n}\right) =\widetilde{S}%
\left( \theta ;\widetilde{\alpha }_{n}\right) +T\left( h;\alpha ^{\ast
}\right) +R\left( \theta ,h;\widetilde{\alpha }_{n}\right) .
\end{equation*}%
Our objective is to show that $R\left( \theta ,h;\widetilde{\alpha }%
_{n}\right) $ is negligible when compared to $\widehat{S}\left( \theta ;%
\widetilde{\alpha }_{n}\right) $ from which we deduce that $\widehat{\theta }
$ behaves as the maximizer of $\widetilde{S}\left( \theta ;\widetilde{\alpha
}_{n}\right) $. Define
\begin{equation*}
R_{1}\left( \theta ,h;\alpha \right) =\dfrac{1}{n}\sum\limits_{i=1}^{n}\left[
\psi \left( Y_{i},\hat{r}_{h}^{i}\left( Z_{i}^{T}\theta ;\theta \right)
;\alpha \right) -\psi \left( Y_{i},r\left( Z_{i}^{T}\theta ;\theta \right)
;\alpha \right) \right] I_{A}\left( Z_{i}\right) \;
\end{equation*}%
and use Taylor expansion to write
\begin{eqnarray*}
R\left( \theta ,h;\widetilde{\alpha }_{n}\right) &=&R_{1}\left( \theta ,h;%
\widetilde{\alpha }_{n}\right) -R_{1}\left( \theta _{0},h;\alpha ^{\ast
}\right) \\
&=&\left[ R_{1}\left( \theta ,h;\widetilde{\alpha }_{n}\right) -R_{1}\left(
\theta _{0},h;\widetilde{\alpha }_{n}\right) \right] +\left[ R_{1}\left(
\theta _{0},h;\widetilde{\alpha }_{n}\right) -R_{1}\left( \theta
_{0},h;\alpha ^{\ast }\right) \right] .
\end{eqnarray*}%
Apply Lemma \ref{rr} to obtain the order of $R_{1}\left( \theta ,h;%
\widetilde{\alpha }_{n}\right) -R_{1}\left( \theta _{0},h;\widetilde{\alpha }%
_{n}\right) .$ Next, note that $R_{1}\left( \theta _{0},h;\widetilde{\alpha }%
_{n}\right) -R_{1}\left( \theta _{0},h;\alpha ^{\ast }\right) $ does not
depend on $\theta $. Deduce that
\begin{eqnarray*}
R\left( \theta ,h;\widetilde{\alpha }_{n}\right) \!\! &=&\!\!\left[
\!O_{P}\left( h^{4}\right) +O_{P}\left( \frac{1}{nh^{2}}\right) +O_{P}\left(
\frac{h^{2}}{\sqrt{n}}\right) +O_{P}\left( \frac{1}{n\sqrt{n}h^{4}}\right) \!%
\right] \times O_{P}\left( \left\| \theta -\theta _{0}\right\| \right) \\
&&+\left[ \!O\left( h^{2}\right) +O_{P}\left( \frac{1}{\sqrt{n}h^{2}}\right)
\!\right] \times O_{P}\left( \left\| \theta -\theta _{0}\right\| ^{2}\right)
\\
&&+\left\{ \text{terms not depending on }\theta \right\} ,
\end{eqnarray*}%
uniformly in $h\in \left[ n^{-(1/2-\varepsilon )},\,n^{-\varepsilon }\right]
,$ with $0<\varepsilon <1/2,$ uniformly in $\theta \in \Theta _{n}$ and
uniformly with respect to $\left\{ \widetilde{\alpha }_{n}\right\} .$ It
follows that, up to terms not containing $\theta ,$
\begin{equation*}
R\left( \theta ,h;\widetilde{\alpha }_{n}\right) =o_{P}\left( \left\| \theta
-\theta _{0}\right\| /\sqrt{n}\right) +o_{P}\left( \left\| \theta -\theta
_{0}\right\| ^{2}\right) ,
\end{equation*}%
uniformly in $h\in \mathcal{H}_{n}$, $\theta \in \Theta _{n}$ and uniformly
with respect to $\left\{ \widetilde{\alpha }_{n}\right\} .$

Now, write $\widetilde{S}\left( \theta ;\widetilde{\alpha }_{n}\right) =%
\widetilde{S}_{1}\left( \theta ;\widetilde{\alpha }_{n}\right) -\widetilde{S}%
_{1}\left( \theta _{0};\alpha ^{\ast }\right) $ with
\begin{equation*}
\widetilde{S}_{1}\left( \theta ,\alpha \right) =\dfrac{1}{n}%
\sum\limits_{i=1}^{n}\psi \left( Y_{i},r\left( Z_{i}^{T}\theta ;\theta
\right) ;\alpha \right) \ I_{A}\left( Z_{i}\right) .
\end{equation*}%
Notice that
\begin{equation*}
\widetilde{S}\left( \theta ,\widetilde{\alpha }_{n}\right) =\left[
\widetilde{S}_{1}\left( \theta ,\widetilde{\alpha }_{n}\right) -\widetilde{S}%
_{1}\left( \theta _{0},\widetilde{\alpha }_{n}\right) \right] +\left[
\widetilde{S}_{1}\left( \theta _{0},\widetilde{\alpha }_{n}\right) -%
\widetilde{S}_{1}\left( \theta _{0},\alpha ^{\ast }\right) \right] ,
\end{equation*}%
where the last difference does not contain $\theta $, so that
\begin{equation*}
\arg \max_{\theta }\widetilde{S}\left( \theta ,\widetilde{\alpha }%
_{n}\right) =\arg \max_{\theta }[\widetilde{S}_{1}\left( \theta ,\widetilde{%
\alpha }_{n}\right) -\widetilde{S}_{1}\left( \theta _{0},\widetilde{\alpha }%
_{n}\right) ].
\end{equation*}%
Furthermore, for any $\widetilde{\alpha }_{n}\rightarrow \alpha ^{\ast },$
in probability, we have $\partial _{\theta \theta ^{T}}^{2}\widetilde{S}%
_{1}\left( \theta _{0},\widetilde{\alpha }_{n}\right) -\partial _{\theta
\theta ^{T}}^{2}\widetilde{S}_{1}\left( \theta _{0},\alpha ^{\ast }\right)
=o_{P}\left( 1\right) .$ Therefore, using the Taylor expansion we can write
\begin{eqnarray*}
&&\!\!\!\!\!\!\!\!\!\!\!\!\widetilde{S}_{1}\left( \theta ,\widetilde{\alpha }%
_{n}\right) -\widetilde{S}_{1}\left( \theta _{0},\widetilde{\alpha }%
_{n}\right) \\
&=&\left( \theta -\theta _{0}\right) ^{T}\,\partial _{\theta }\widetilde{S}%
_{1}\left( \theta _{0},\widetilde{\alpha }_{n}\right) +\left( \theta -\theta
_{0}\right) ^{T}\,\partial _{\theta \theta ^{T}}^{2}\widetilde{S}_{1}\left(
\theta _{0},\widetilde{\alpha }_{n}\right) \,\left( \theta -\theta
_{0}\right) +o_{P}\left( \left\| \theta -\theta _{0}\right\| ^{2}\right) \\
&& \\
&=&\left( \theta -\theta _{0}\right) ^{T}\,\partial _{\theta }\widetilde{S}%
_{1}\left( \theta _{0},\alpha ^{\ast }\right) +\left( \theta -\theta
_{0}\right) ^{T}\,\left[ \partial _{\theta }\widetilde{S}_{1}\left( \theta
_{0},\widetilde{\alpha }_{n}\right) -\partial _{\theta }\widetilde{S}%
_{1}\left( \theta _{0},\alpha ^{\ast }\right) \right] \\
&&+\left( \theta -\theta _{0}\right) ^{T}\,\partial _{\theta \theta ^{T}}^{2}%
\widetilde{S}_{1}\left( \theta _{0},\alpha ^{\ast }\right) \,\left( \theta
-\theta _{0}\right) \\
&&+\left( \theta -\theta _{0}\right) ^{T}\,\left[ \partial _{\theta \theta
^{T}}^{2}\widetilde{S}_{1}\left( \theta _{0},\widetilde{\alpha }_{n}\right)
-\partial _{\theta \theta ^{T}}^{2}\widetilde{S}_{1}\left( \theta
_{0},\alpha ^{\ast }\right) \right] \,\left( \theta -\theta _{0}\right) \\
&&+o_{P}\left( \left\| \theta -\theta _{0}\right\| ^{2}\right) \\
&& \\
&=&\!\!\frac{1}{\sqrt{n}}\ \left( \theta -\theta _{0}\right)
^{T}\,V_{n}+\left( \theta -\theta _{0}\right) ^{T}\left[ \partial _{\theta }%
\widetilde{S}_{1}\left( \theta _{0},\widetilde{\alpha }_{n}\right) -\partial
_{\theta }\widetilde{S}_{1}\left( \theta _{0},\alpha ^{\ast }\right) \right]
\!-\left( \theta -\theta _{0}\right) ^{T}\,W_{n}\,\left( \theta -\theta
_{0}\right) \\
&&+o_{P}\left( \left\| \theta -\theta _{0}\right\| ^{2}\right) ,
\end{eqnarray*}%
uniformly in $\theta $ in $o_{P}\left( 1\right) $ neighborhoods of $\theta
_{0}$, where
\begin{eqnarray*}
V_{n} &=&\dfrac{1}{\sqrt{n}}\sum\limits_{i=1}^{n}\partial _{\theta }\psi
\left( Y_{i},r\left( Z_{i}^{T}\theta _{0};\theta _{0}\right) ;\alpha ^{\ast
}\right) I_{A}\left( Z_{i}\right) , \\
W_{n} &=&-\ \dfrac{1}{n}\sum\limits_{i=1}^{n}\partial _{\theta \theta
^{T}}^{2}\psi \left( Y_{i},r\left( Z_{i}^{T}\theta _{0};\theta _{0}\right)
;\alpha ^{\ast }\right) I_{A}\left( Z_{i}\right)
\end{eqnarray*}%
(here, $\partial _{\theta }\widetilde{S}_{1}$ and $\partial _{\theta }\psi $
are vectors in $\mathbb{R}^{d-1},$ while $\partial _{\theta \theta ^{T}}^{2}%
\widetilde{S}_{1}$ and $\partial _{\theta \theta ^{T}}^{2}\psi $ are $%
(d-1)\times (d-1)$ matrices). Next, write
\begin{equation*}
\partial _{\theta }\widetilde{S}_{1}\left( \theta _{0},\widetilde{\alpha }%
_{n}\right) -\partial _{\theta }\widetilde{S}_{1}\left( \theta _{0},\alpha
^{\ast }\right) =\partial _{\alpha }\partial _{\theta }\widetilde{S}%
_{1}\left( \theta _{0},\overline{\alpha }\right) \left[ \widetilde{\alpha }%
_{n}-\alpha ^{\ast }\right]
\end{equation*}%
($\partial _{\alpha }$ denote the partial derivative with respect to $\alpha
)$ with $\overline{\alpha }$ between $\widetilde{\alpha }_{n}$ and $\alpha
^{\ast }$, and notice that by the definition of $\psi $ as the logarithm of
a LEFN density, for any $\alpha ,$
\begin{equation*}
E\left[ \partial _{\alpha }\partial _{\theta }\widetilde{S}_{1}\left( \theta
_{0},\alpha \right) \right] =E\left[ \partial _{\theta }\partial _{\alpha
}\psi \left( Y,r\left( Z^{T}\theta _{0};\theta _{0}\right) ;\alpha \right)
I_{A}\left( Z\right) \right] =0.
\end{equation*}%
Consequently, $\partial _{\alpha }\partial _{\theta }\widetilde{S}_{1}\left(
\theta _{0},\alpha \right) =O_{P}\left( 1/\sqrt{n}\right) ,$ uniformly in $%
\alpha .$ Deduce that for all $\widetilde{\alpha }_{n}\rightarrow
\alpha ^{\ast },$ in probability, $\partial _{\theta
}\widetilde{S}_{1}\left( \theta _{0},\widetilde{\alpha }
_{n}\right) -\partial _{\theta }\widetilde{S}_{1}\left( \theta
_{0},\alpha ^{\ast }\right) =o_{P}\left( 1/\sqrt{n}\right) .$ Use
this fact and the order of $R\left( \theta ,h;\widetilde{\alpha }
_{n}\right) $ to write
\begin{eqnarray}
\widehat{S}\left( \theta ,h;\widetilde{\alpha }_{n}\right) &=&\dfrac{1}{%
\sqrt{n}}\left( \theta -\theta _{0}\right) ^{T}V_{n}-\frac{1}{2}\left(
\theta -\theta _{0}\right) ^{T}\ W_{n}\ \left( \theta -\theta _{0}\right) \
\label{quad1} \\
&&+o_{P}\left( \left\| \theta -\theta _{0}\right\| /\sqrt{n}\right)
+o_{P}\left( \left\| \theta -\theta _{0}\right\| ^{2}\right) +\left\{ \text{%
terms not depending on }\theta \right\} ,  \notag
\end{eqnarray}%
uniformly in $h\in \mathcal{H}_{n}$, $\theta \in \Theta _{n}$ and $%
\widetilde{\alpha }_{n}$ in $o_{P}\left( 1\right) $ neighborhoods of $\alpha
^{\ast }.$ By little algebra,
\begin{equation*}
I=E\left[ \partial _{\theta }\psi \left( Y,r_{\theta _{0}}\left( Z^{T}\theta
_{0}\right) ;\alpha ^{\ast }\right) \ \partial _{\theta }\psi \left(
Y,r_{\theta _{0}}\left( Z^{T}\theta _{0}\right) ;\alpha ^{\ast }\right)
^{T}I_{A}\left( Z\right) \right]
\end{equation*}%
\begin{equation*}
J=-E\left[ \partial _{\theta \theta }^{2}\psi \left( Y,r_{\theta _{0}}\left(
Z^{T}\theta _{0}\right) ;\alpha ^{\ast }\right) \ I_{A}\left( Z\right) %
\right] .
\end{equation*}%
Deduce from the assumptions that $V_{n}$ converges in distribution to $%
\mathcal{N}(0,I)$ and $W_{n}\rightarrow J,$ in probability. Finally, deduce
that $\widehat{\theta }$ has the same asymptotic distributions as the
maximizer of the quadratic form (\ref{quad1}). More precisely, apply
Theorems 1 and 2 of Sherman (1994a) to obtain first, the $\sqrt{n}-$%
consistency of $\widehat{\theta }$ and next, the asymptotic normality
\begin{equation*}
\sqrt{n}\left( \widehat{\theta }-\theta _{0}\right) \overset{\mathcal{D}}{%
\longrightarrow }\mathcal{N}\left( 0,J^{-1}IJ^{-1}\right) .
\end{equation*}

\emph{Part II : the behavior of }$\widehat{h}.$\emph{\ }By Taylor expansion
we can write
\begin{equation*}
T\left( h;\alpha ^{\ast }\right) =T_{0}+T_{1}(h;\alpha ^{\ast
})+T_{2}(h;\alpha ^{\ast })+\left\{ \text{negligible terms}\right\} ,
\end{equation*}
where $T_{0}$ is independent of $h$,
\begin{equation*}
T_{1}(h;\alpha ^{\ast })=\dfrac{1}{n}\sum\limits_{i=1}^{n}\ \partial
_{2}\psi \left( Y_{i},r\left( Z_{i}^{T}\theta _{0};\theta _{0}\right)
;\alpha ^{\ast }\right) \left[ \hat{r}_{h}^{i}\left( Z_{i}^{T}\theta
_{0};\theta _{0}\right) -r\left( Z_{i}^{T}\theta _{0};\theta _{0}\right) %
\right] \ I_{A}\left( Z_{i}\right) ,
\end{equation*}
\begin{equation*}
T_{2}(h;\alpha ^{\ast })=\dfrac{1}{n}\sum\limits_{i=1}^{n}\frac{1}{2}\
\partial _{22}^{2}\psi \left( Y_{i},r\left( Z_{i}^{T}\theta _{0};\theta
_{0}\right) ;\alpha ^{\ast }\right) \left[ \hat{r}_{h}^{i}\left(
Z_{i}^{T}\theta _{0};\theta _{0}\right) -r\left( Z_{i}^{T}\theta _{0};\theta
_{0}\right) \right] ^{2}\ I_{A}\left( Z_{i}\right) .
\end{equation*}
By Lemma \ref{mise},  $T_{2}(h;\alpha ^{\ast
})=-C_{1}h^{4}-C_{2}/nh+o_{P}(h^{4}+1/nh)$,
uniformly over $\mathcal{H}_{n},$ with $C_{1},$ $C_{2}$ defined in (\ref%
{a1a2}). Moreover, by Lemma \ref{mise2}, $T_{1}(h;\alpha ^{\ast
})=o_{P}(T_{2}(h;\alpha ^{\ast })),$ uniformly over $\mathcal{H}_{n}.$
Finally, recall that
\begin{equation*}
R(\theta ,h;\widetilde{\alpha }_{n})=\left[ R_{1}\left( \theta ,h;\widetilde{%
\alpha }_{n}\right) -R_{1}\left( \theta _{0},h;\widetilde{\alpha }%
_{n}\right) \right] +\left[ R_{1}\left( \theta _{0},h;\widetilde{\alpha }%
_{n}\right) -R_{1}\left( \theta _{0},h;\alpha ^{\ast }\right) \right] .
\end{equation*}
The order of $R_{1}\left( \theta _{0},h;\widetilde{\alpha }_{n}\right) $ and
$R_{1}\left( \theta _{0},h;\alpha ^{\ast }\right) $ can be obtained in the
same way as the order of $T\left( h;\alpha ^{\ast }\right) .$ Taking the
differences $R_{1}\left( \theta _{0},h;\widetilde{\alpha }_{n}\right)
-R_{1}\left( \theta _{0},h;\alpha ^{\ast }\right) $ vanishes the constants
of the dominating terms containing $h$. Thus, up to terms independent of $h$%
, the second bracket is negligible compared with $T\left( h;\alpha ^{\ast
}\right) $, uniformly in $\left\{ \widetilde{\alpha }_{n}\right\} .$ On the
other hand, by Lemma \ref{rr}, the first bracket is of order $%
o_{P}(T_{2}(h;;\alpha ^{\ast }))$, uniformly in $\theta $ in $%
O_{P}(n^{-1/2}) $ neighborhoods of $\theta _{0}$ and $h\in \mathcal{H}_{n}$
and uniformly in $\left\{ \widetilde{\alpha }_{n}\right\} $ in $o_{P}\left(
1\right) $ neighborhood of $\alpha ^{\ast }.$ Since $\widehat{\theta }$ was
shown to be $\sqrt{n}-$consistent, deduce that $\widehat{h}$ is
asymptotically equivalent to the maximizer of $T_{2}(h;\alpha ^{\ast }).$
More precisely, $\widehat{h}/h_{n}^{opt}\rightarrow 1,$ in probability,
where $h_{n}^{opt}=(C_{2}/4C_{1})^{1/5}n^{-1/5}.$

To close the proof it remains to show that $(\widehat{\theta },\widehat{h})$
defined in (\ref{deeff}) is asymptotically equivalent to the maximizer of
the objective function $\left( \theta ,h\right) \rightarrow \widehat{S}%
\left( \theta ,h;\widetilde{\alpha }_{n}\right) $ in equation (\ref{infeasi}%
). For this we use inequality (\ref{ikj}) and Lemma \ref{idic}. Moreover, we
can decompose $\widehat{S}\left( \theta ,h;\widetilde{\alpha }_{n},A^{\delta
}\right) $ in the same way as $\widehat{S}\left( \theta ,h;\widetilde{\alpha
}_{n}\right) $ and obtain the same orders, uniformly over $\Theta _{n}\times
\mathcal{H}_{n},$ uniformly with respect to $\left\{ \widetilde{\alpha }%
_{n}\right\} $ and uniformly in $\delta \in \lbrack 0,\delta
_{0}]$, for some small $\delta _{0}$. Note that $A^{\delta }$
shrinks to the set $\left\{ z:f\left( z^{T}\theta _{0};\theta
_{0}\right) =c\right\} $ as $\delta \rightarrow 0.$ Therefore,
the constants appearing in the dominating terms of the decomposition of $%
\widehat{S}\left( \theta ,h;\widetilde{\alpha }_{n},A^{\delta }\right) $
vanishes as $\delta \rightarrow 0$, provided that $P\left[ f\left(
Z^{T}\theta _{0};\theta _{0}\right) =c\right] =0$. Consequently, the $%
O_{P}\left( \cdot \right) $ orders are transformed in $o_{P}\left( \cdot
\right) $ orders and thus $\widehat{S}\left( \theta ,h;\widetilde{\alpha }%
_{n},A^{\delta }\right) =o_{P}(\widehat{S}\left( \theta ,h;\widetilde{\alpha
}_{n},A\right) ),$ uniformly in $\theta \in \Theta _{n}$, $h\in \mathcal{H}%
_{n}$ and $\widetilde{\alpha }_{n}$ a sequence convergent to $\alpha ^{\ast
},$ in probability, provided that $\delta \rightarrow 0$. The proof is
complete.
\end{proof}

\quad

\begin{center}
REFERENCES
\end{center}

\begin{description}
\item {\footnotesize \textsc{Andrews, D.W.K.} (1988). Chi-Square diagnostic
tests for econometric models: theory. \textsl{Econo\-metrica}, \textbf{56},
1419-1453. }

\item {\footnotesize \textsc{Andrews, D.W.K.} (1995). Nonparametric kernel
estimation for semiparametric models.\ \textsl{Econometric Theory}, \textbf{11}, 560-596. }

\item {\footnotesize \textsc{Cameron, A.C., and Trivedi, P.K. }(2013).\
\textit{Regression Analysis of Count Data. Second Edition}. Econometric Society Monographs,
Cambridge University Press.\ \ }


\item {\footnotesize \textsc{Cui, X., H\"{a}rdle, W. and Zhu, L.} (2011). The EFM approach for single-index models. \textsl{Ann. Statist.}, \textbf{39},  1658--1688. }

\item {\footnotesize \textsc{Delecroix, M., Hristache, M. and
Patilea, V.} (2006). On semiparametric M-estimation in
single-index regression. \textsl{J. Statist. Plann. Inference}, \textbf{136}, 730--769 }


\item {\footnotesize \textsc{Gouri\'{e}roux, C. Monfort, A. and Trognon, A.}
(1984a). Pseudo maximum likelihood methods: theory. \textsl{Econometrica},
\textbf{52}, 681-700. }

\item {\footnotesize \textsc{Gouri\'{e}roux, C. Monfort, A. and Trognon, A.}
(1984b). Pseudo maximum likelihood methods: applications to Poisson models.
\textsl{Econometrica}, \textbf{52}, 701-720. }

\item {\footnotesize \textsc{H\"{a}rdle, W., Hall, P. and Ichimura, H.}
(1993). Optimal smoothing in single-index models. \textsl{Ann. Statist.},
\textbf{21}, 157-178. }

\item {\footnotesize \textsc{H\"{a}rdle, W. and Marron, J.S.} (1985).
Optimal bandwidth selection in nonparametric regression function estimation.
\textsl{Ann. Statist.}, \textbf{13}, 1465-1481. }

\item  {\footnotesize \textsc{H\"{a}rdle, W. }and \textsc{Stoker,
T.M.} (1989). Investigating smooth multiple regression by the
method of average derivatives. \textsl{J. Amer. Statist. Assoc.},
\textbf{84}, 986-995.}

\item {\footnotesize \textsc{Ichimura, H. }(1993). Semiparametric least
squares (SLS) and weighted SLS estimation of single-index models. \textsl{J.
Econometrics}, \textbf{58}, 71-120. }

\item {\footnotesize \textsc{Johnson, N.L., Kotz, S. and Balakrishnan N.}\
(1997).\ \textit{Discrete Multivariate Distributions}. New York, John
Wiley.\ }


\item {\footnotesize \textsc{Newey, W.K. }(1993). Efficient
estimation of models with conditional moment restrictions, in
G.S.\ Maddala, C.R. Rao and H.D. Vinod (eds.) \textit{Handbook of
Statistics, vol.\ 11, }pp. 419- 454, New-York: North-Holland. }

\item  {\footnotesize \textsc{Newey, W.K.} (1994). The asymptotic
variance of semiparametric estimators. \textsl{Econometrica},
\textbf{62}, 1349-1382. }

\item {\footnotesize \textsc{Newey, W.K. and Stoker, T.M. }(1993).
Efficiency of weighted average derivative estimators and index models.
\textsl{Econometrica}, \textbf{61}, 1199-1223. }

\item {\footnotesize \textsc{Newey, W.K. and McFadden, D. }(1994).\ Large
sample estimation and hypothesis testing, in R.F.\ Engle and D.L. McFadden
(eds.) \textit{Handbook of Econometrics, vol.\ IV, }pp. 2111- 2245,
New-York: North-Holland. }

\item {\footnotesize \textsc{Pakes, A. and Pollard, D.} (1989). Simulation
and the asymptotics of optimization estimators. \textsl{Econometrica},
\textbf{57}, 1027-1057. }

\item {\footnotesize \textsc{Powell, J.L., Stock, J.M.
}and\textsc{\ Stoker, T.M.} (1989). Semiparametric estimation of
index coefficients. \textsl{Econometrica}, \textbf{57},
1403-1430.}

\item {\footnotesize \textsc{Picone, G.A. and Butler, J.S.} (2000).\
Semiparametric estimation of multiple equation models.\ \textsl{Econometric
Theory}, \textbf{16}, 551-575. }

\item {\footnotesize \textsc{Sellar, C., Stoll, J.R. and Chavas, J.P. }
(1985). Validation of empirical measures of welfare change: a comparison of
nonmarket techniques. \textsl{Land Economics}, \textbf{61}, 156-175.}

\item {\footnotesize \textsc{Sherman, R.P.} (1994a). Maximal inequalities
for degenerate $U-$processes with applications to optimization estimators.
\textsl{Ann. Statist.}, \textbf{22}, 439-459. }

\item {\footnotesize \textsc{Sherman, R.P.} (1994b). $U$-processes in the
analysis of a generalized semiparametric regression estimator. \textsl{
Econometric Theory}, \textbf{10}, 372-395. }

\item {\footnotesize \textsc{Xia, Y. and Li, W.K. }(1999). On
single-index coefficient regression models. \textsl{J. Amer.
Statist. Assoc.,} \textbf{94}, 1275-1285. }

\item {\footnotesize \textsc{Xia, Y., Tong, H. and Li, \negthinspace
W.K.\negthinspace } (1999). On extended partially linear single-index
models. \textsl{Biometrika}, \textbf{86}, 831-842. }

\item {\footnotesize \textsc{Xia, Y., Tong, H., Li, W.K.} and\ \textsc{Zhu,
L.-X.} (2002) An adaptive estimation of dimension reduction space (with
discussions). \textsl{J. R. Statist. Soc. }B, \textbf{64}, 363--410. }

\item{\footnotesize \textsc{Ziegler, A.} (2011). \textsl{Generalized Estimating Equations.} Lecture Notes in Statistics, Volume 204. New-York: Springer. }

\end{description}

{\footnotesize \ }

\end{document}